\documentclass[10pt,a4paper]{article}

\usepackage[left=3.1cm, right=3.1cm, bottom=3cm]{geometry} 

\usepackage[T1]{fontenc} 
\usepackage{ngerman} 

\usepackage[english]{babel}

\newcommand{\footremember}[2]{%
	\footnote{#2}
	\newcounter{#1}
	\setcounter{#1}{\value{footnote}}%
}
 
\usepackage{palatino}

\usepackage{amsmath}
\usepackage{amsthm}
\usepackage{amssymb}
\usepackage{url}
\usepackage{natbib}
\usepackage{hyperref}
\usepackage{graphicx}
\usepackage{url}
\usepackage{listings}

\makeindex

\allowdisplaybreaks

\newtheorem{Theorem}{Theorem}[section]
\newtheorem{Definition}[Theorem]{Definition} 
\newtheorem{Lemma}[Theorem]{Lemma}	
\newtheorem{Proposition}[Theorem]{Proposition}
\newtheorem{Corollary}[Theorem]{Corollary}

\numberwithin{equation}{section} 
\newcommand{\C}{\ensuremath{\mathbb{C}}}
\newcommand{\R}{\ensuremath{\mathbb{R}}}

\newcommand{\E}{\ensuremath{\mathbb{E}}}
\newcommand{\PP}{\ensuremath{\mathbb{P}}}
\newcommand{\F}{\ensuremath{\mathbb{F}}}
\newcommand{\Z}{\ensuremath{\mathbb{Z}}}
\newcommand{\N}{\ensuremath{\mathbb{N}}}

\newcommand{\Div}{\ensuremath{\operatorname{div}}}
\newcommand{\curl}{\ensuremath{\operatorname{curl}}}

\newcommand{\expbt}{\ensuremath{e^{-i\sum_{l=1}^{N}B_l\beta_l(t) }}}
\newcommand{\expbbt}{\ensuremath{e^{i\sum_{l=1}^{N}B_l\beta_l(t) }}}

\newcommand{\expb}{\ensuremath{e^{-i\sum_{l=1}^{N}B_l\beta_l(s) }}}
\newcommand{\expbb}{\ensuremath{e^{i\sum_{l=1}^{N}B_l\beta_l(s) }}}

\newcommand{\ft}{\ensuremath{\mathcal{F}_t}}

\newcommand{\ind}{\ensuremath{\mathbf{1}}}

\newcommand{\re}{\operatorname{Re}}

\DeclareMathOperator*{\esssup}{\operatorname{ess\,sup}}

\begin{document}
\title{Strong solutions to a nonlinear stochastic Maxwell equation with a retarded material law}

\author{Luca Hornung\footremember{address}{{Institute for Analysis, Karlsruhe Institute of Technology, Englerstr. 2, D-76128 Karlsruhe, Germany (luca.hornung@kit.edu)}} }

\date{\today}
        \maketitle
     
    \begin{abstract}
    	\noindent
We study the Cauchy problem for a semilinear stochastic Maxwell equation with Kerr-type nonlinearity and a retarded material law. We show existence and uniqueness of strong solutions using a refined Faedo-Galerkin method and spectral multiplier theorems for the Hodge-Laplacian. We also make use of a rescaling transformation that reduces the problem to an equation with additive noise to get an appropriate a priori estimate for the solution.

\medskip


    \end{abstract}

    \section{Introduction}
In this article, we consider the semilinear stochastic Maxwell equation 
\begin{equation}\label{stochastic_maxwell_intro}
\begin{cases}
du(t)&=\big[Mu(t)-|u(t)|^qu(t)+(G\ast u)(t)+J(t)\big]\operatorname{dt}+[b(t)+B(t,u(t))]dW(t),\\
u(0)&=u_0
\end{cases}
\end{equation}
in $  L^{2}(D)^6=L^{2}(D)^3\times L^{2}(D)^3 $ driven by a cylindrical Brownian motion $ W(t) $ with the retarded material law $$ (G\ast u)(t)=\int_{0}^{t}G(t-s)u(s)\operatorname{ds} $$ and the perfect conductor boundary condition $ u_1\times \nu=0 $ on $ \partial D. $ Here, the Maxwell operator is given by
\[ M\binom{u_1}{u_2}=\binom{\curl u_2}{-\curl u_1} \]
for $ 3d $ vector fields $ u_1 $ and $ u_2. $ We allow $ D $ to be a bounded domain or $ D $ might also be the full-space $ \R^{3} $ (in this case the boundary condition drops).

This equation is a model for a stochastic electromagnetic system in weakly-nonlinear chiral media and was derived in \cite{roach_stratis_yannacopoulus} in chapter 2. It originally comes from the deterministic Maxwell system
\begin{equation*}
\begin{cases}
\partial_t(Lu(t))&=Mu(t)+J(t)\\
u(0)&=u_0
\end{cases}
\end{equation*}
with constitutive relation \[ Lu(t,x)=\kappa(x)u(t,x)+\int_{0}^{t}K_1(t-s,x)u(s,x)\operatorname{ds}+\int_{0}^{t}K_2(t-s,x)|u(s,x)|^qu(s,x)\operatorname{ds}. \]
This material law consists of an instantaneous part $ \kappa u $ with a hermitian, uniformly positive definite and uniformly bounded matrix $ \kappa:D\to\C^{6\times 6} $, a linear dispersive part $ K_1\ast u $ and of a nonlinear dispersive part $ K_2\ast |u|^qu. $ This power-type nonlinearity is motivated by the Kerr-Debye model. Note, that in applications, one would only take the nonlinearity $ |u_1|^qu_1 $ or $ |u_2|^qu_2 $ to model a nonlinear polarisation or magnetisation. We take the two quantities together to study both phenomena at once. Using the product rule, we end up with
\begin{equation*}
\begin{cases}
\kappa u'&=Mu-K_1(0)u-K_2(0)|u|^qu-(\partial_tK_1)\ast u-(\partial_tK_2)\ast|u|^qu+J\\
u(0)&=u_0
\end{cases}
\end{equation*}
At this point, we introduce additional simplifications. We assume that the term $ (\partial_tK_2)\ast|u|^qu $ can be neglected. This is typical for a weakly nonlinear medium since one assumes that both the dispersion and the nonlinear effects are weak, so that the combination then satisfies $ (\partial_tK_2)\ast|u|^qu<< K_2(0)|u|^qu. $ Usually one demands $ K_1(0):D\to\C^{6\times 6} $ to be bounded and positive semi-definite and $ K_2(0):D\to \C^{6\times 6} $ to be bounded and uniformly positive definite. But for sake of simplicity, we choose $ K_1(0)\equiv 0 $ and $ K_2(0)\equiv I. $ We just note that the results are unchanged by this simplification and the proofs could be adjusted easily. Moreover, we choose $ \kappa=I $. We must admit that this simplification is necessary at this point since our methods cannot deal with coefficients so far. The problems one has to overcome if $ \kappa\neq I $ are discussed in section $ 6 $ in detail. Setting $ G:=-\partial_tK_1, $ we get a deterministic version of \eqref{stochastic_maxwell_intro}.

In many applications, there is some uncertainty concerning the external sources or the precise behaviour of the medium itself. In these cases, it is useful to model $ u $ as random variables on a probability space $ \Omega $ and impose a stochastic noise perturbation. Here one distinguishes between the additive noise $ b $ perturbing $ J $ and the multiplicative noise $ B(u) $ perturbing the medium. A linear stochastic version of \eqref{stochastic_maxwell_intro} was already discussed in \cite{roach_stratis_yannacopoulus}, chapter $ 12. $ Moreover, in \cite{chen_zhang_preservation_of_physical_properties_stochastic_maxwell}, the authors show that typical conservation laws of linear electromagnetic system are preserved under additive noise perturbation. The authors in \cite{hong_ji_zhang_numeric_stochastic_maxwell} also treat linear stochastic Maxwell equations numerically with energy-conserving methods. More about the application of random media in scattering, wave propagation and in the theory of composites can be found in \cite{bal_waves_in_random_media}, \cite{field2008electromagnetic}, \cite{fouque_wave_propagation_random_media} and \cite{milton_the_theory_of_composites}.

However, as far as we know, there are no known results about a nonlinear stochastic Maxwell equation. One reason might be that in the absence of Strichartz estimates for $ (e^{tM} )_{t\in\R}$, even local solvability is a tricky issue. Moreover, there is no embedding of the $ D(M)\hookrightarrow L^{p} $, that helps to control the nonlinearity.
Even the deterministic version of \eqref{stochastic_maxwell_intro} has not beed treated rigorously so far. In \cite{roach_stratis_yannacopoulus}, the authors profess to prove well-posedness, but their argument ignores some severe complications. Since they claim to have better deterministic results than ours, we discuss their approach in section 6 in detail.

Now, we briefly sketch our strategy. At first, we show in section $ 4, $ that \eqref{stochastic_maxwell_intro} has a unique weak solution 
\begin{equation}\label{stochastic_maxwell_intro_weak_regularity}
 u\in L^{2}(\Omega;C(0,T;L^{2}(D)))^6\cap L^{q+2}(\Omega\times[0,T]\times D)^6.
\end{equation}
 This is done in two steps. First, we use a version of the Galerkin method from R\"ocker and Pr\'evot (see \cite{prevot_rockner_a_concise_course_on_stochastic_pde}) to solve \eqref{stochastic_maxwell_intro} in the special case $ G\equiv 0 $ and make use of the monotone structure of our nonlinearity. As this is approach is well-known, we just discuss the different steps and concentrate on how to deal with the additional term $ Mu $, despite the fact, that $ u\notin D(M). $ Afterwards, we inflict the retarded material law with Banach's fixed point theorem. 

The proof of the existence and uniqueness of a strong solution, that additionally satisfies
\[ Mu\in L^{2}(\Omega;L^{\infty}(0,T;L^{2}(D)))^6+ L^{\frac{q+2}{q+1}}(\Omega\times[0,T]\times D)^6 \]
is more tricky. Again, we start with $ G\equiv 0 $ and we add a nontrivial $ G $ at the very end. In a deterministic setting, one would try to estimate $ \|u'(t)\|_{L^{2}(D)^6} $ and then use \eqref{stochastic_maxwell_intro_weak_regularity} to control $ Mu. $ However, solutions of stochastic differential equations are not differentiable in time. The first idea was to derive an estimate for $ \|Mu(t)-|u(t)|^qu(t)+J(t)\|_{L^{2}(D)^6}^2$ with Gronwall's Lemma, but we failed since the It\^o formula for this quantity contains the term $$ \|D_{vv} (|v|^qv)(u(t))\big(B(u(t)),B(u(t)\big)\|_{L^{2}(D)^6}^2, $$ we could not estimate properly. Hence, we choose the noise $ \sum_{j=1}^{N}\big(b_j(t)+iB_ju(t)\big)d\beta_j(t) $ and use the rescaling transform $$ y(t)=u(t)e^{-i\sum_{j=1}^{N}B_j\beta_j(t)} $$
to get rid of the multiplicative noise in the same way as Barbu and R\"ockner in \cite{barbu_rockner_random_scaled_porous_media} and \cite{barbu_rockner_stochastic_variation_inequalities} (see also \cite{BarbuH1},\cite{BarbuL2}) and \cite{barbu_rockner_zhang_stochastic_nonlinear_schrodinger}). The difference in our approach is that the authors have natural a priori estimates before transforming the equation and just transform to solve the transformed equation with deterministic techniques. Moreover, they just use multiplicative noise. We use the transform to get better a priori estimates and consider an equation that also has additive noise. The arising equation has the form
\begin{equation*}
(\operatorname{TSEE})\begin{cases}
dy(t)&=[My(t)-|y(t)|^{q}y(t)+A(t)y(t)+\widetilde{J}(t)]\operatorname{dt}+\sum_{i=1}^{N}\widetilde{b}_i(t)\ d\beta_i(t),\\
u(0)&=u_0,
\end{cases}
\end{equation*}
with a nonautonomous operator $ A(t) $ having random coefficients. We truncate $ (\operatorname{TSEE}) $ with a refined Faedo-Galerkin approach, i.e. we solve
\begin{equation*}
\begin{cases}
dy_n(t)&=[P_nMy_n(t)-P_n|y_n(t)|^{q}y_n(t)+P_nA(t)y_n(t)+P_n\widetilde{J}(t)]\operatorname{dt}+\sum_{i=1}^{N}S_{n-1}\widetilde{b}_i(t)\ d\beta_i(t),\\
y_n(0)&=S_{n-1} u_0.
\end{cases}
\end{equation*}
Here, $ P_n=\ind_{[0,2^n]}(-\Delta_H) $ and $ S_n=\psi(-2^{-n}\Delta_H) $ for some $ \psi\in C_c^{\infty}(D) $ with $ \operatorname{supp} \psi\subset [0,2] $ and $ \psi= 1 $ on $ [0,1] $ are spectral multipliers with respect to the Hodge-Laplacian $ \Delta_H $ on $ L^{p} $, that is the component-wise Laplacian with domain
\begin{align*}
 \Big\{(u_1,u_2)\in L^{p}(D)^6: &\curl u_1,\curl u_2, \curl\curl u_1,\curl\curl u_2\in L^{p}(D)^3, \Div u_1\in W_0^{1,p}(D),\\
 &\Div u_2\in W^{1,p}(D),u_1\times\nu =0,u_2\cdot\nu=0,\curl u_2\times \nu=0\text{ on }\partial D \Big \}.
\end{align*}
Now $ P_n$ and $S_n $ reduce the problem to an ordinary stochastic differential equation that can be solved easily. Moreover, we show that $ P_n,S_n $ are self-adjoint on $ L^{2}(D)^6 $ and commute with both $ \Delta_H $ and $ M $. Further, we have $ \|S_nu\|_{L^{p}(D)^6}\leq C  \|u\|_{L^{p}(D)^6}$ with a constant $ C>0 $ depending on $ p $, but not on $ u $ and $ n. $ Note that such an estimate is not available for $ P_n $ in a general situation. This remarkable uniform $ L^{p} $-boundedness is a consequence of \cite{kunstmann_uhl_spectral_multiplier_theorems}, together with generalized Gaussian bounds for the Hodge-Laplacian (see \cite{mitrea_monniaux_on_the_analyticity_stokes}, \cite{kunstmann_uhl_spectral_multipliers_for_some_elliptic_systems}). The deep connection between $ \Delta_H $ and $ M $ is a consequence of the formula
\[ -\Delta_H=\curl\curl-\operatorname{grad}\Div, \]
which implies $ \Delta_H=M^2 $ in the range of the Helmholtz projection $ P_H $ and $ M^2=0 $ in the range of $ (I-P_H). $ This interplay will be examined in detail in section $ 3. $ The idea to use spectral multiplier results in such a way was firstly used by Brzezniak, F. Hornung and Weis in \cite{brzensiak_hornung_weis_martingale_solutions_for_nls}. Afterwards, we estimate $$ \|P_nMy_n(t)-P_n|y_n(t)|^{q}y_n(t)+P_nA(t)y_n(t)+P_n\widetilde{J}(t)\|^2_{L^{2}(D)^6} $$ using It\^o's
formula, the monotone structure of the equation and the properties of $ P_n,S_n $. This yields the desired estimate for $ My_n $ uniformly in $ n. $ Finally, we pass to the limit again using the monotonicity of the nonlinearity and undo the transformation. 

In section $ 6 $, we explain how the result changes if one strengthens some of the assumptions and we discuss interesting special cases, such as the deterministic version of \eqref{stochastic_maxwell_intro}, $ b\equiv 0 $ or a constant $ B. $ Moreover, we sketch a program to extend this approach to non-constant coefficients $ \kappa\neq I $.

    \section{Preliminaries}
The purpose of this section is to provide a short overview over the basic tools used in this paper. For most of the proofs and further details, we give references to the literature.

Throughout this paper, let $ (\Omega,\mathfrak{F},\F=(\ft)_{t\geq 0},\PP) $ be a filtered probability space that satisfies the usual conditions, i.e. $ \mathcal{F}_0 $ contains all $ \PP $-null sets and the filtration is right-continuous. Moreover, given normed spaces $ X $ and $ Y $, $ B(X,Y) $ denotes the set of all linear and bounded operators from $ X $ to $ Y $. Further, we write $ C(a,b;X) $ for the space of uniformly continuous functions on $ [a,b] $ with values in $ X $ equipped with its usual norm and $ L^{2}(H_1,H_2) $ for the space of Hilbert-Schmidt operators between the Hilbert spaces $ H_1 $ and $ H_2. $ Throughout this article, $ D\subset \R^{3}  $ will either be a bounded $ C^1 $-domain or $ D=\R^{3}. $ If we evaluate a function on $ \partial D, $ this always corresponds to the first case and has no meaning in the second case.
\subsection{\texorpdfstring{The operators $ \curl $ and $ \Div $}{The operators curl and div}}
First, we give a short introduction into vector calculus. To motivate the definition of functions with vanishing tangential component or normal component on the boundary, we make the following calculation with smooth functions $ f,g:\overline{D}\to\R^3. $ Using vector calculus and the Divergence theorem, we obtain
\begin{align*}
\int_{\partial D}f\cdot(g \times\nu)\operatorname{d\sigma}&=\int_{\partial D}\nu\cdot(f \times g)\operatorname{d\sigma}=\int_{D}\Div(f\times g)(x)\operatorname{dx}\\
&=\int_{D}\curl g(x)\cdot f(x)\operatorname{dx}-\int_{D}g(x)\cdot\curl f(x)\operatorname{dx}.
\end{align*}
 Similarly, we get
\begin{align*}
\int_{\partial D}y(z \cdot\nu)\operatorname{d\sigma}&=\int_{D}\Div(y\cdot z)(x)\operatorname{dx}\\
&=\int_{D}\nabla y(x)\cdot z(x)\operatorname{dx}+\int_{D}y(x)\Div z(x)\operatorname{dx}.
\end{align*}
for smooth $ y:\overline{D}\to\R $ and $ z:\overline{D}\to\R^{3}. $ 
Hence, we can define vanishing tangential and normal components on the boundary in a natural way.
\begin{Definition}\label{stochastic_maxwell_trace}
	Let $ D\subset\R^{3} $ be bounded $ C^1 $-domain with boundary $ \partial D $ and $ p\in [1,\infty). $
	\begin{itemize}
\item[a)] Let $ g\in L^{p}(D)^3 $ with $ \curl g\in L^{p}(D)^3. $ We say $ g\times \nu=0 $ on $ \partial D, $ if
\[ \int_{D}\curl\phi(x)\cdot g(x)\operatorname{dx}=\int_{D}\phi(x)\cdot\curl g(x)\operatorname{dx} \]
for every $ \phi\in C^{\infty}(\overline{D})^3 $.
\item[b)] Let $ z\in L^{p}(D)^3 $ with $ \Div z\in L^p(D). $ We say $ z\cdot \nu=0 $ on $ \partial D, $ if
\[ \int_{D}\nabla \phi(x)\cdot z(x)\operatorname{dx}=-\int_{D}y(x)\Div z(x)\operatorname{dx} \]
for every $ \phi\in C^{\infty}(\overline{D}) $.
	\end{itemize}
\end{Definition}
Next, we introduce the subspaces of $ L^{2}(D)^3 $ associated with $ \curl $ and $ \Div. $
\begin{Definition}We set
\begin{itemize}
\item [a)] $ H(\curl)(D):=\big\{u\in L^{2}(D)^3:\ \curl u\in L^{2}(D)^3 \big\}. $ 
\item [b)] $ H(\curl,0)(D):=\big\{u\in H(\curl)(D):\ u\times\nu=0\ \operatorname{on}\ \partial D \big\}. $
\item[c)] $ H(\Div)(D):=\big\{u\in L^{2}(D)^3:\ \Div u\in L^{2}(D) \big\}. $ 
\item [d)] $ H(\Div,0)(D):=\big\{u\in H(\Div)(D):\ u\cdot\nu=0\ \operatorname{on}\ \partial D \big\}. $ 
\end{itemize}
\end{Definition}
We define the Maxwell operator $ M $ with perfect conductor boundary condition by
\[ M\binom{u_1}{u_2}=\binom{\curl u_2}{-\curl u_1} \]
on the domain $ D(M)=H(\curl,0)(D)\times H(\curl)(D). $
\begin{Proposition}\label{stochastic_maxwell_maxwell_skew}
The Maxwell operator $ M $ is skew-adjoint on $ L^{2}(D)^6 $, i.e. we have
\[ \int_{D}My(x)\cdot z(x)\operatorname{dx}=-\int_{D}y(x)\cdot Mz(x)\operatorname{dx} \]
for every $ y,z\in D(M) $ and $ D(M)=D(M^{*}). $
\end{Proposition}
\begin{proof}
This result is well-known. See e.g. \cite{hochbruck_schnaubelt_jahnke_convergence_splitting_maxwell}, section $ 3. $ 
\end{proof}
The next technical Lemma will be needed later on.  We state it for functions in the sum of $ L^{p} $-spaces for technical reasons. This will only be necessary, when $ D=\R^3. $ 
\begin{Lemma}\label{maxwell_stochastic_lemma_distributional_maxwell_operator}
Let $ D $ be a bounded $ C^1 $- domain or $ D=\R^3 $, $ y\in L^{2}(D)^6 $ and $ p\in [1,\infty) $. If there exists $ z\in L^{2}(D)^6+L^{p}(D)^6, $ such that 
\begin{equation}\label{maxwell_stochastic_lemma_distributional_maxwell_operator_eq}
\int_{D}y(x)\cdot M\phi(x)\operatorname{dx}=-\int_{D}z(x)\cdot\phi(x)\operatorname{dx}
\end{equation}
for every $ \phi\in C^{\infty}(\overline{D})^6\cap L^{2}(D)^6\cap L^{\frac{p}{p-1}}(D)^6 $ with $ M\phi\in L^{2}(D)^6 $ and $ \phi_1\times \nu=0 $ on $ \partial D, $ we have $ My=z $ in the sense of distributions and $ y_1\times\nu=0 $ on $ \partial D. $
\end{Lemma}
\begin{proof}
By inserting $ \phi=(\phi_1,0) $ and $ \phi=(0,\phi_2) $ into \eqref{maxwell_stochastic_lemma_distributional_maxwell_operator_eq}, we derive
\begin{align*}
\int_{D}y_2(x)\cdot \curl\phi_1(x)\operatorname{dx}&=\int_{D}z_1(x)\cdot\phi_1(x)\operatorname{dx}\\
\int_{D}y_1(x)\cdot \curl\phi_2(x)\operatorname{dx}&=-\int_{D}z_2(x)\cdot\phi_2(x)\operatorname{dx}
\end{align*} 
for any smooth $ \phi_1 $ with $ \phi_1\times \nu=0 $ on $ \partial D $ and for any smooth $ \phi_2. $ Inserting $ \phi_1,\phi_2\in C_c^{\infty}(D)^3 $ yields $ \curl y_2=z_1 $ and $ \curl y_1=-z_2 $ in the sense of distributions, i.e. $ My=z $ in the sense of distributions. If $ D\neq\R^{d} $, we have to show the claimed boundary condition. The second identity implies 
\begin{align*}
\int_{D}y_1(x)\cdot \curl\psi(x)\operatorname{dx}+\int_{D}\curl y_1(x)\cdot\psi(x)\operatorname{dx}=0
\end{align*}
for every $ \psi\in C^{\infty}(\overline{D})^3$ and hence, $ y_1\times\nu=0 $ on $ \partial D $ in the sense of Definition \ref{stochastic_maxwell_trace}.
\end{proof}
\subsection{\texorpdfstring{The power nonlinearity $ |u|^qu $}{The power nonlinearity}}
In this subsection, we mention the basic properties of nonlinearity $ u\mapsto F(u)=|u|^qu $ as a mapping from $ L^{q+2}(D)^6 $ to $ L^{\frac{q+2}{q+1}}(D)^6 $ with $ q>0 $. We start with its monotonicity.
\begin{Lemma}\label{stochastic_maxwell_properties_nonlinearity}
	$ F $ satisfies the estimate
	\begin{equation}\label{stochastic_maxwell_properties_nonlinearity_monotonicity}
	\int_{D} \re\langle F(v)(x)-F(u)(x),u(x)-v(x)\rangle_{\C^6}\operatorname{dx}\leq -C\|u-v\|^{q+2}_{L^{q+2}(D)^6}
	\end{equation}
	for some $ C>0 $ and for all $ u,v\in L^{q+2}(D)^6. $
\end{Lemma}
\begin{proof}
Clearly, $ \|F(u)\|_{L^{\tfrac{q+2}{q+1}}(D)^6}=\|u\|_{L^{q+2}(D)^6} $ and therefore $ F $ has the claimed mapping properties. The estimate \eqref{stochastic_maxwell_properties_nonlinearity_monotonicity} is a direct consequence of Lemma $ 4.4 $ in \cite{dibeneddeto_degenerate_parabolic_equations}.
\end{proof}
Since we often use It\^o's formula, we need to know the differentiability properties of $ F. $
\begin{Lemma}\label{stochastic_maxwell_nonlinearity_properties}
	The nonlinearity $ F:L^{q+2}(D)^6\to L^{\frac{q+2}{q+1}}(D)^6 $, $ u\mapsto |u|^qu $ is real Fr\'echet continuously differentiable with $ \re\langle F'(u)v,v\rangle_{L^2(D)^6}\geq 0 $ and
	\[ |F'(u)v(x)|\lesssim |u(x)|^{q}|v(x)| \]
		for all $ u,v\in L^{q+2}(D)^6 $ and $ x\in D $. In particular, it is locally Lipschitz continuous, i.e.
		\[ \|F(u)-F(v)\|_{L^{\frac{q+2}{q+1}}(D)^6}\lesssim\big(\|u\|_{L^{q+2}(D)^6}^q+\|q+2\|_{L^{2}(D)^6}^q\big)\|u-v\|_{L^{q+2}(D)^6}. \]
	Moreover, if $ q\in (1,\infty) $, it is twice real continuously differentiable  with 
	\begin{align*}
	&F''(u)(v,v)(x)\lesssim |u(x)|^{q-1}|v(x)|^2
	\end{align*}
	for all $ u,v\in L^{q+2}(D)^6. $ 
\end{Lemma}
\begin{proof}
	It is well-known, that $ F:L^{q+2}(D)^6\to L^{\frac{q+2}{q+1}}(D)^6 $ is real differentiable with
	\[ F'(u)v=q|u|^{q-2}\re\langle u,v\rangle_{\C^{6}}u+|u|^qv \]
	for every $ u,v\in L^{q+2}(D)^6 $. See e.g. given \cite{schnaubelt_isem_dispersive_equations}, Corollary $ 9.3 $
	
	 Consequently, we also have
	\begin{align*}
	\re\big\langle F'(u)v,v\big\rangle_{L^{2}(D)^6}=\int_{D}q|u(x)|^{q-2}\big(\re\langle u(x),v(x)\rangle_{\C^{6}}\big)^2+|u(x)|^{q}|v(x)|^2\operatorname{dx}\geq 0.
	\end{align*}
	 Moreover, we estimate
	\begin{align*}
	F'(u)v(x)\leq C|u(x)|^{q}|v(x)|
	\end{align*}
	for some $ C>0 $. For the second derivative, we start with formal calculation for $ F'' $ and get
	\begin{align*}
	F''(u)(v,w)=&q|u|^{q-2}\Big((q-2)|u|^{-2}\re\langle u,w\rangle_{L^{2}(D)^6}\re\langle u,v\rangle_{L^{2}(D)^6}u+\re\langle w,v\rangle_{L^{2}(D)^6}u\\
	&+ \re\langle u,w\rangle_{L^{2}(D)^6}v+\re\langle u,v\rangle_{L^{2}(D)^6}w \Big)
	\end{align*}
	
	For sake of readability, we do not rigorously show that $ F:L^{q+2}(D)^6 \to L^{q+2}(D)^6 $ is twice Fr\'echet differentiable with this derivative. However, to give an impression how to show this, we check that 
	last term in $ F'(u)v $, namely $ u\mapsto [v\mapsto |u|^qv]:L^{q+2}(D)^6\to B(L^{q+2}(D),L^{\frac{q+2}{q+1}}(D)^6) $, is Fr\'echet differentiable with derivative $ G(u)(v,w)=q|u|^{q-2}\re\langle u,w\rangle_{\C^{6}}v $. Let $ u,v,w\in L^{q+2}(D)^6 $ with $ v,w\neq 0. $  Then, H\"olders inequality together with the mean value theorem yield
	\begin{align*}
	\big\||u|^qv-&|u+w|^qv-G(u)(v,w)\big\|_{L^{\frac{q+2}{q+1}}(D)^6}\\
	&\leq \big\||u|^q-|u+w|^q-q|u|^{q-2}\re\langle u,w\rangle_{\C^{6}}\big\|_{L^{\frac{q+2}{q}}(D)^6}\|v\|_{L^{q+2}(D)^6}\\
	&\lesssim \big\|\int_{0}^{1}\re\langle |u+\theta w|^{q-2}(u+\theta w)-|u|^{q-2}u,w\rangle_{\C^{6}} \operatorname{d\theta}\big\|_{L^{\frac{q+2}{q}}(D)^6}\|v\|_{L^{q+2}(D)^6}\\
	&\leq\int_{0}^{1}\big\||u+\theta w|^{q-2}(u+\theta w)-|u|^{q-2}u\big\|_{L^{\frac{q+2}{q-1}}(D)^6}\operatorname{d\theta}\|w\|_{L^{q+2}(D)^6}\|v\|_{L^{q+2}(D)^6}
	\end{align*}
	Hence, we showed
	\begin{align}\label{stochastic_maxwell_frechet_diff_nonlinear}
	\|w\|_{L^{q+2}(D)^6}^{-1}\big\|v\mapsto |u|^qv-&|u+w|^qv-G(u)(v,w)\big\|_{B(L^{q+2}(D)^6,L^{\frac{q+2}{q+1}}(D)^6)}\notag\\
	&\lesssim \int_{0}^{1}\big\||u+\theta w|^{q-2}(u+\theta w)-|u|^{q-2}u\big\|_{L^{\frac{q+2}{q-1}}(D)^6}\operatorname{d\theta}
	\end{align}
	for all  $ u,w\in L^{q+2}(D)^6 $ with $ w\neq 0. $ 
	
	It remains to prove that this quantity tends to $ 0 $ as $ w\to 0 $ in $ L^{q+2}(D)^6. $
	Let $ (w_n)_n $ be a sequence in $ L^{q+2}(D)^6 $ with $ w_n\to 0 $ as $ n\to\infty $ and let $ (w_{n_k})_{k} $ be an arbitrary subsequence. Hence there exists another subsequence, still denoted with $ (w_{n_k})_{k} $, such that $ w_{n_k}\to 0 $ almost everywhere for $ k\to\infty $ and such that $ |w_{n_k}|\leq g $ for some $ g\in L^{q+2}(D)^6 $. We also have $$ |u+\theta w_{n_k}|^{q-2}(u+\theta w_{n_k})-|u|^{q-2}u\to 0 $$ almost everywhere as $ k\to\infty. $ Together with the bound 
	\[ \Big||u+\theta w_{n_k}|^{q-2}(u+\theta w_{n_k})-|u|^{q-2}u\Big|\leq |u|^{q-1}+|w_{n_k}|^{q-1}\leq |u|^{q-1}+g^{q-1}, \]
	for $ \theta\in [0,1] $ and the fact that $ u\in L^{q+2}(D)^6, $ we get
	\[ \int_{0}^{1}\big\||u+\theta w_{n_k}|^{q-2}(u+\theta w_{n_k})-|u|^{q-2}u\big\|_{L^{\frac{q+2}{q-1}}(D)^6}\operatorname{d\theta}\to 0 \]
	as $ k\to\infty $. All in all, this shows that the left hand side of \ref{stochastic_maxwell_frechet_diff_nonlinear} tends to $ 0 $ as $ w\to 0 $ and we established the Fr\'echet differentiability of $ u\mapsto [v\mapsto |u|^qv] $ with derivative $ G(u). $ The claimed estimate for $ F''(u)(v,v)(x) $ is immediate. This closes the proof.
\end{proof}
\section{\texorpdfstring{The Hodge-Laplacian on a bounded $ C^1 $-domain and its spectral multipliers}{The Hodge-Laplacian on a bounded domain}}
In this section, we introduce the Hodge-Laplace operator on a bounded $ C^1 $-domain $ D $, and we define the spectral projections needed in the sequel. We consider the bilinear form
\begin{align*}
a(u,v)=\int_{D}(\curl u)(x)(\curl v)(x)\operatorname{dx}+\int_{D}(\Div u)(x)(\Div v)(x)\operatorname{dx}
\end{align*}
with $ D(a) $ either given by $ V^{(1)}:= H(\operatorname{curl},0)(D)\cap H(\operatorname{div})(D) $ or by $ V^{(2)}:= H(\operatorname{curl})(D)\cap H(\operatorname{div},0)(D) $ equipped with the norm
\begin{align*}
\|u\|^2_{V^{(i)}}:=\|\curl u\|_{L^{2}(D)}^2+\|\Div u\|_{L^{2}(D)}^2+\|u\|_{L^{3}(D)^2}
\end{align*}
for $ i=1,2. $ In both cases, $ a $ is symmetric and bounded. Moreover, $ a $ is coercive in sense
\[ a(u,u)=\|u\|_{V^{(i)}}^2-\|u\|_{L^{2}(D)}^2 \]
for all $ u\in V^{(i)} $, $ i=1,2 $. Setting 
\begin{align*}
D(A^{(1)})&=\{u\in V^{(1)}:\curl\curl u\in L^{2}(D)^3,\ \Div u\in W^{1,2}_0(D) \}\\
D(A^{(2)})&=\{u\in V^{(2)}:\curl\curl u\in L^{2}(D)^3,\ \curl u\times \nu=0\text{ on }\partial D,\ \Div u\in W^{1,2}(D) \}
\end{align*}
it turns out, that $ a $ with $ D(a)=V^{(1)} $ is associated with the operator $ A^{(1)}=-\Delta $ on the domain $ D(A^{(1)}), $ whereas $ a $ with $ D(a)=V^{(2)} $ is associated with the operator $ A^{(2)}=-\Delta $ on the domain $ D(A^{(2)}). $ To see this, use partial integration for $ \curl $ and $ \Div $ and the formula $ \curl\curl=\operatorname{grad}\Div-\Delta $. By the coercivity of the corresponding forms, the operators $ I+A^{(i)} $, $ i=1,2 $, are strictly positive. Moreover, the symmetrie implies that they are self-adjoint on $ L^{2}(D)^3 $ and since the embeddings $ V^{(i)}\hookrightarrow L^{2}(D)^3 $ are compact (see \cite{amrouche_bernadi_dauge_girault_vector_potential_in_three-dimensional_domains}, Theorem $ 2.8 $), the embeddings $ D(A^{(i)})\hookrightarrow L^{2}(D)^3 $ are also compact. Consequently, there exists two orthonormal basis of eigenvectors $ (h_j^{(i)})_{j\in\N} $ to the positive eigenvalues $ (\lambda_j^{(i)})_{j\in\N} $ of $ I+A^{(i)} $ with $ \lambda_j^{(i)}\to\infty $ for $ j\to\infty. $

The next Proposition shows, that these operators satisfy generalized Gaussian estimates. We add an additional sectral shift, since some of the theorems we apply require strictly positive operators. 
\begin{Proposition}\label{stochastic_maxwell_hodge_laplacian_gaussian_bounds}
	Both $ I+A^{(1)} $ and $ I+A^{(2)} $ satisfy generalized Gaussian $ (2,q) $ estimates for every $ q\in [2,\infty)$, i.e. for every $ q\in [2,\infty) $ there exists $ C,b>0 $, such that 
	\begin{align*}
\|\ind_{B(x,t^\frac{1}{2})}e^{-t(I+A^{(i)})}\ind_{B(y,t^\frac{1}{2})}\|_{B(L^{2}(D)^3,L^{q}(D)^3}\leq Ct^{-\frac{3}{2}(\frac{1}{2}-\frac{1}{q})}e^{-\frac{b|x-y|^2}{t}}
	\end{align*}
	for all $ t>0 $ and all $ x,y\in D $.
\end{Proposition}
\begin{proof}
	In \cite{kunstmann_uhl_spectral_multipliers_for_some_elliptic_systems}, the authors argue on page $ 239, $ that both $ A^{(1)} $ and $ A^{(2)} $ satisfy generalized Gaussian $ (2,q) $-bounds for every $ q\in [2,q_{D}). $ Here, $ q_{D}\in [2,\infty) $ denotes the supremum over all indexes $ p $ for which the boundary value problems
	\begin{equation*}
	\begin{cases}
	\Delta u=f \text{ in }D,\\
	\curl u,\ \curl\curl u\in L^{p}(D)^3,\ \Div(u)\in W^{1,p}(D),\\
	u\cdot\nu=0,\ \curl(u)\times\nu =0 \text{ on }\partial D
	\end{cases}
	\end{equation*}
	and
	\begin{equation*}
	\begin{cases}
	\Delta u=f \text{ in }D,\\
	\curl u,\ \curl\curl u\in L^{p}(D)^3,\ \Div(u)\in W^{1,p}_0(D),\\
	u\times\nu=0 \text{ on }\partial D
	\end{cases}
	\end{equation*}
	have a unique solution. This argument heavily makes use of iterative resolvent estimate for the Hodge-Laplacian (see \cite{mitrea_monniaux_on_the_analyticity_stokes}, section $ 5 $ and $ 6 $). By \cite{mitrea_sharp_hodge_decomposition}, Theorem $ 1.2 $ and $ 1.3, $ we know that $ q_D=\infty $ since $ D $ is a $ C^1 $-domain in $ \R^{3}. $ Finally note that Gaussian estimates are preserved under positive spectral shifts.
\end{proof}
For more details about these operators, we refer to \cite{mitrea_monniaux_on_the_analyticity_stokes}, where they are discussed in a more general differential geometric setting.

We define spectral multipliers with the natural functional calculus for self-adjoint operators having a basis of eigenvectors. Let $ \Psi\in C_c^{\infty}(\R) $ with $ \operatorname{supp}(\Psi)\subset [\tfrac{1}{2},2] $ and $ \sum_{l\in\Z}\Psi(2^{-l}x)=1 $ for all $ x>0 $. The operators $ P_n:L^{2}(D)^6\to L^{2}(D)^6$ and $ S_n:L^{2}(D)^6\to L^{2}(D)^6 $ are defined by
\begin{align*}
P_n(u)=\binom{\ind_{[0,2^n]}(I+A^{(1)})(u_1)}{\ind_{[0,2^n]}(I+A^{(2)})(u_2)},\quad
S_n(u)=\binom{\sum_{l=-\infty}^{n}\Psi(2^{-l}(I+A^{(1)}))(u_1)}{\sum_{l=-\infty}^{n}\Psi(2^{-l}(I+A^{(2)}))(u_2)}
\end{align*}
for $ u=(u_1,u_2)\in L^{2}(D)^6 $ and $ n\in\N. $ Note, that the above sums are finite, since only finitely many eigenvalues of $ A^{(i)} $ are smaller than $ n. $ The next Proposition summarizes the most important properties of $ S_n $ and $ P_n $ as operators on $ L^{2}(D)^6 $.

\begin{Proposition}\label{stochastic_maxwell_spectral_multiplier_l2_properties}
	$ P_n $ and $ S_n $ satisfy
	\begin{itemize}
		\item[i)] $ P_n $ is a projection, i.e. $ P_n^2=P_n $.
		\item[ii)] The operators $ P_n,S_n  $ are self-adjoint with $ \|P_n\|_{B(L^{2}(D)^6)}=\|S_n\|_{B(L^{2}(D)^6)}=1 $ for every $ n\in\N $.
		\item [iii)] $ P_n$ and $S_m $ commute for every $ n,m\in\N. $
		\item [iv)] The range of $ P_n $ and $ S_n $ is finite dimensional. Moreover, we have $ R(P_n),R(S_n)\subset D(M) $ for every $ n\in\N. $
		\item[v)]We have $ R(S_{n-1})\subset R(P_n)\subset R(S_{n}) $, $ S_nP_n=P_n $ and $ P_nS_{n-1}=S_{n-1}. $
		\item[vi)] We have $ \lim_{n\to\infty}P_nx=\lim_{n\to\infty}S_nx=x $ for every $ x\in L^{2}(D)^6 $.
	\end{itemize}
\end{Proposition} 
\begin{proof}
We have $ \sum_{l=-\infty}^{n}\Psi(2^{-l}\cdot)=\ind_{(0,2^n)}+\psi(2^{-n}\cdot)\ind_{[2^n,2^{n+1})} $, by choice of $ \psi. $ Hence, all these properties follow from the functional calculus for self-adjoint operators in Hilbert spaces.
\end{proof}
Moreover, the operators $ S_n $ have the following property, that will be crucial in what follows.
\begin{Lemma}\label{stochastic_maxwell_properties_galerkin_uniformly_in_n}
	For every $ p\in (1,\infty) $, the operators
	$ S_n$ are operators from $ L^{p}(D)^6 $ to $ L^{p}(D)^6 $ with a bound depending on $ p $, but not on $ n\in\N $. Moreover, we have $ S_nf\to f $ in $ L^{p}(D)^6 $ as $ n\to\infty $ for all $ f\in L^{p}(D)^6.$
\end{Lemma}
\begin{proof}
	The first statement follows from the spectral multiplier theorem $ 5.4 $ in \cite{kunstmann_uhl_spectral_multiplier_theorems} as a consequence of the generalized Gaussian bounds for $ A^{(1)} $ and $ A^{(2)}. $ One could also argue with the more general Theorem $ 7.1 $ in \cite{kriegler_weis_spectral_multiplier_via_functional_calculus}. 
The claimed convergence property is then a special case from \cite{kriegler_weis_paley_littlewood_decomposition_for_sectorial_operators_and_interpolation_spaces}, Theorem $ 4.1. $ To apply this Theorem the $ 0 $-sectoriality of $ -\Delta_H $ and the boundedness of a Mikhlin functional calculus $ \mathcal{M}^\alpha $ in $ L^{p}(D)^6 $ for some $ \alpha>0 $ are needed. The first is checked in \cite{mitrea_monniaux_on_the_analyticity_stokes}, Theorem $ 6.1, $ whereas the second holds true with $ \alpha>4 $ by the generalized Gaussian bounds (see \cite{kriegler_weis_paley_littlewood_decomposition_for_sectorial_operators_and_interpolation_spaces}, Lemma $ 6.1 $, $ (3) $).
\end{proof}
Next, we introduce two different Helmholtz projections on $ L^{2}(D)^3 $. The proof for the following statement is well-known and can be found amongst others in \cite{hettlich_kirsch_the_mathematical_theory_of_time_harmonic_maxwells_equations}, section $ 4.1.3. $
\begin{Proposition}\label{maxwell_stochastic_helmholtz}
Let $ D\subset\R^{3} $ be a bounded Lipschitz domain. Given $ u\in L^{2}(D)^3, $ the following decompositions hold true. 
\begin{itemize}
\item[$ (1) $] There exists a unique $ p\in W^{1,2}_0(D) $ and $ \widetilde{u}\in H(\operatorname{div})(D) $ with $ \Div\widetilde{u}=0 $ such that $ u=\widetilde{u}+\nabla p. $ The corresponding operator $ P^{(1)}_H:L^{2}(D)^3\to L^{2}(D)^3, u\mapsto\widetilde{u} $ is an orthogonal projection.
\item[$ (2) $] There exists a unique $ p\in W^{1,2}(D) $ with $ \int_{D}p(x)\operatorname{dx}=0 $ and $ \widetilde{u}\in H(\Div,0)(D) $ with $ \Div\widetilde{u}=0 $ such that $ u=\widetilde{u}+\nabla p. $ The corresponding operator $ P^{(2)}_H:L^{2}(D)^3\to L^{2}(D)^3, u\mapsto\widetilde{u} $ is an orthogonal projection.
\end{itemize}
In particular, $ P_H(u_1,u_2):=\big(P_H^{(1)}u_1,P_H^{(2)}u_2\big) $ for $ u_1,u_2\in L^{2}(D)^3 $ defines an orthogonal projection on $ L^{2}(D)^6 $.
\end{Proposition}
To simplify the notation in what follows, we combine $ A^{(1)} $ and $ A^{(2)} $ to a self-adjoint operator $ -\Delta_H (u_1,u_2):=(A^{(1)}u_1,A^{(2)}u_2) $ for $ (u_1,u_2)\in D(A^{(1)})\times D(A^{(2)}) $. 
The Helmholtz projection $ P_H $ is closely related to both $ M $ and $ \Delta_H $. In the following Lemma, we exploit the fact $ M^2=\Delta_H $ on $ D(M)\cap P_H(L^{2}(D)^6) $ to show some powerful identities.
\begin{Lemma}\label{stochastic_maxwell_commutation_of_M_and_S_n}
We have $ P_H\Delta_H=\Delta_HP_H $ on $ D(\Delta_H) $, $ MP_H=P_HM $ on $ D(M) $ and $ P_nM=MP_n $, $ S_nM=MS_n $ on $ D(M). $
\end{Lemma}
\begin{proof}
	The first claim can be found in \cite{mitrea_monniaux_on_the_analyticity_stokes}, section $ 3 $ or in \cite{kunstmann_uhl_spectral_multipliers_for_some_elliptic_systems}, Lemma $ 5.4. $ Consequently, we also have $ S_nP_H=P_HS_n $ and $ P_nP_H=P_HP_n $, since $ S_n $ and $ P_n $ are in the functional calculus of $ \Delta_H $. 
	
	For the second statement, we first show that $ M=P_HM $. Due to $ \Div\curl =0, $ we just have to show $ \curl u_1\cdot\nu=0 $ on $ \partial D $ for $ u_1\in H(\curl,0)(D). $ Definition \ref{stochastic_maxwell_trace} $ a) $ yields
	\[ \int_{D}\nabla\phi(x)\cdot\curl u_1(x)\operatorname{dx}=\int_{D}\curl\nabla\phi(x)\cdot u_1(x)\operatorname{dx}=0=\int_{D}\phi(x)\Div\curl u_1(x)\operatorname{dx}, \]
	for every $ \phi\in C^{\infty}(\overline{D}) $, which implies $ \curl u_1\cdot\nu=0 $ according to Definition \ref{stochastic_maxwell_trace} $ b). $ As a consequence of $ \curl\nabla =0, $ we know $ M(I-P_H)=0. $ All in all we get
	\[ MP_H-P_HM=MP_H-M=M(I-P_H)=0. \]
	
	Finally, the identity
	$$ \Delta_H=\binom{-\curl\curl+\operatorname{grad}\Div}{-\curl\curl+\operatorname{grad}\Div}=M^2 $$ 
	on $ D(M^2)\cap P_H(L^{2}(D)^6)=  D(\Delta_H)\cap P_H(L^{2}(D)^6) $ together with $ M(I-P_H)=0 $ imply
	\begin{align*}
		MP_n=MP_H\ind_{[0,2^n]}(-\Delta_H)=M\ind_{[0,2^n]}(-M^2)P_H=\ind_{[0,2^n]}(-M^2)MP_H =P_nM
	\end{align*}
	on $ D(M). $ For $ S_nM=MS_n, $ one may argue analogously.
	\end{proof}
\begin{Corollary}\label{stochastic_maxwell_density}
$ \bigcup_{n=1}^{\infty} R(P_n) $ is dense in $ D(M) $ and in $ L^{p}(D)^6 $ for any $ p\in (1,\infty). $
\end{Corollary}
\begin{proof}
Let $ u\in D(M). $ Using the commutation property of $ P_n $ from Lemma \ref{stochastic_maxwell_commutation_of_M_and_S_n}, we get
\[ \|Mu-MP_nu\|_{L^{2}(D)^6}=\|Mu-P_nMu\|_{L^{2}(D)^6}\xrightarrow{n\to\infty} 0.\]
If on the other hand $ u\in L^{p}(D)^6, $ we get $ S_nu\to u $ in $ L^{p}(D)^6 $ from Lemma \ref{stochastic_maxwell_properties_galerkin_uniformly_in_n}. This together with Proposition \ref{stochastic_maxwell_spectral_multiplier_l2_properties} $ v) $ proves the claimed result.
\end{proof}

We also consider \eqref{stochastic_maxwell_intro} on $ \R^{3} $ and hence, we need an analogue to the $ P_n $ and $ S_n $ we defined above. However, in the absence of boundary conditions, things are far more easy. We define
\[ P_nf=S_nf:=\mathcal{F}^{-1}\big(\xi\mapsto \ind_{[-2^n,2^n]}(\xi_1)\ind_{ [-2^n,2^n]}(\xi_2)\ind_{[-2^n,2^n]} (\xi_3)\hat{f}(\xi)\big) \]
for $ f\in L^{2}(D)^6. $ As $ M $ is a differential operator, it commutes with this frequency cut-off. Moreover, $ P_n, S_n $ satisfy the same properties as in Propositions \ref{stochastic_maxwell_spectral_multiplier_l2_properties} expect $ iv) $. Further, as a consequence of the boundedness of the Hilbert transform on $ L^{p}(\R^{3}), $ they are bounded on $ L^{p}(\R^{3})^6 $. This finally results in an analogue to Lemma \ref{stochastic_maxwell_properties_galerkin_uniformly_in_n} and Corollary \ref{stochastic_maxwell_density}.
For details, we refer to \cite{Grafakos_classical}, chapter $ 6.1.3. $
We end this section with a Lemma showing mapping properties of $P_n $ as operator between $ L^{2}(D)^6 $ and $ L^{p}(D)^6. $
\begin{Lemma}\label{stochastic_maxwell_properties_galerkin}
	For fixed $ n\in\N $, $ p\in [2,\infty) $ and $ q\in (1,2] $, the operator
	$ P_n:L^{q}(D)^6\to L^{2}(D)^6 $ and $ P_n:L^{2}(D)^6\to L^{p}(D)^6 $ is linear and bounded.
\end{Lemma}
\begin{proof}
	This is trivial, if $ D $ is bounded, since all norms on a finite dimensional space are equivalent. In the other case, it is sufficient to show that $ P_n:L^{q}(\R^{3})^6\to L^{2}(\R^{3})^6 $ is bounded, the rest then follows by duality. 
%
The H\"older and the Hausdorff-Young inequality yield
\begin{align*}
\|P_nf\|_{L^{2}(\R^{3})^6}&=\|\xi\mapsto\ind_{[-2^n,2^n]}(\xi_1)\ind_{ [-2^n,2^n]}(\xi_2)\ind_{[-2^n,2^n]} (\xi_3)\hat{f}(\xi)\|_{L^2(\R^{3})^6}\lesssim_{n}\|\hat{f}\|_{L^{\tfrac{q}{q-1}}(\R^{3})^6}\\
&\leq \|f\|_{L^{q}(\R^{3})^3}. 
\end{align*}
\end{proof}

    \section{Existence and uniqueness of a weak solution}
In this section, we will prove existence and uniqueness of a weak solution in the sense of partial differential equations of 
\begin{equation*}
	(\operatorname{WSEE})\begin{cases}
		du(t)&=\big[Mu(t)-|u(t)|^qu(t)+(G\ast u)(t)+J(t)\big]\operatorname{dt}+B(t,u(t))dW_t,\\
		u(0)&=u_0
	\end{cases}
\end{equation*}
for any $ q> 0 $. For sake of readability, we sometimes write $ F(u):=|u|^qu. $ Before we start, we explain our solution concept.
\begin{Definition}\label{stochastic_maxwell_weak_solution_definition}
	We say that an adapted process $ u:\Omega\times [0,T]\to L^{2}(D)^6 $ with $$ u\in L^{2}(\Omega;C(0,T;L^{2}(D)))^6\cap L^{q+2}(\Omega\times[0,T]\times D)^6  $$ is a weak solution of $ (\operatorname{WSEE}) $, if 
\begin{align*}
\langle u(t)-u_0,\phi\rangle_{L^2(D)^6}
=&\int_{0}^{t} -\big\langle u(s),M\phi\big\rangle_{L^2(D)^6}+\big\langle-|u|^qu+J(s)+(G\ast u)(s),\phi\big\rangle_{L^2(D)^6}\operatorname{ds}\\
&+ \int_{0}^{t}\big\langle B(s,u(s)),\phi\big\rangle_{L^2(D)^6}dW(s).
\end{align*}
holds almost surely for all $ t\in [0,T] $ and for all $ \phi\in D(M)\cap L^{q+2}(D)^6. $ Moreover, we call a weak solution $ u $ unique, for any other weak solution $ v $, there exists $ N\subset \Omega $ with $ \PP(N)=0 $, such that $ u(\omega,t)=v(\omega,t) $ for all $ \omega\in\Omega\setminus N $ and all $ t\in [0,T]. $
\end{Definition}
We make the following assumptions.
\begin{itemize}
	\item [\lbrack\text{W1}\rbrack] Let $ D\subset\R^{3} $ be a bounded $ C^1 $-domain or $ D=\R^{3}. $  
	\item[\lbrack\text{W2}\rbrack] The initial value $ u_0:\Omega\to L^{2}(D)^6 $ is strongly $ \mathcal{F}_0 $- measurable.
	\item[\lbrack\text{W3}\rbrack] Let $ G:\Omega\times[0,T]\to B( L^{2}(D)^6), $ such that $ x\mapsto G(t)x $ is for all $ x\in L^{2}(D)^6 $ strongly measurable and $ \F $-adapted. Moreover, we assume 
	\[  \esssup_{\omega\in\Omega}\int_{0}^{T} \|G(\omega,t)\|_{B( L^{2}(D)^6)}\operatorname{dt}< \infty. \]
	\item[\lbrack\text{W4}\rbrack] Let $ U $ be a separable Hilbert space and $ W $ a $ U $-cylindrical Brownian motion. Moreover, let $ B:\Omega\times [0,T]\times D\times L^{2}(D)^6 \to L^2(U, L^{2}(D)^6) $ be strongly measurable, such that $ \omega\mapsto B(\omega,t,x,u) $ is for almost all $ t\in [0,T],\ x\in D $  and all  $u\in L^{2}(D)^6 $ $ \F $-adapted. Moreover, there exists $ C>0 $, such that $ B $ is of linear growth, i.e.
	\[ \|B(t,u)\|_{L^{2}(U; L^{2}(D)^6)}\leq C\|u\|_{ L^{2}(D)^6} \]
	 and Lipschitz
	\[ \|B(t,u)-B(t,v)\|_{L^{2}(U; L^{2}(D)^6)}\leq C\|u-v\|_{ L^{2}(D)^6} \]
	almost surely for almost all $ t\in [0,T] $ and all $ u,v\in  L^{2}(D)^6. $
	\item[\lbrack\text{W5}\rbrack] $ J:\Omega\times[0,T]\to  L^{2}(D)^6 $ is strongly measurable, $ \F $-adapted and we assume $ J\in L^{2}(\Omega\times[0,T]\times D)^6. $
\end{itemize}
At first, we need an It\^o formula, that is appropriate to deal with weak solutions. Our result is a version of \cite{prevot_rockner_a_concise_course_on_stochastic_pde}, Theorem $ 4.2.5 $, that additionally allows a skew-adjoint operator $M $ in spite of the fact that our weak solution is not in $ D(M) $. Our proof relies on a more straightforward regularization technique than the original using the spectral multipliers $ S_n $ from section $ 2.2 $.
\begin{Lemma}\label{stochastic_maxwell_ito_formula}
	Let $ X_0\in L^{2}(\Omega\times D)^6 $ and $ Y\in L^{\frac{q+2}{q+1}}(\Omega\times [0,T]\times D)^6+L^{2}(\Omega\times[0,T]\times D)^6 $ and $ Z\in L^{2}(\Omega\times[0,T];L^{2}(U;L^{2}(D)^6)) $ be $ \F $-adapted. If
\begin{align*}
 \langle X(t),\phi\rangle_{L^2(D)^6}=&\langle X_0,\phi\rangle_{L^2(D)^6}+\int_{0}^{t}-\langle X(s),M\phi\rangle_{L^2(D)^6} +\langle  Y(s),\phi\rangle_{L^2(D)^6}\operatorname{dt}\\
 &+\int_{0}^{t}\langle Z(s),\phi\rangle_{L^2(D)^6}dW(s)
\end{align*}
	almost surely for all $ t\in [0,T] $ and all $ \phi\in D(M)\cap L^{q+2}(D)^6 $  and we additionally have the regularity $ X\in L^{q+2}(\Omega\times [0,T]\times D)^6\cap L^{2}(\Omega\times[0,T]\times D)^6, $ the It\^o formula
	\begin{align}\label{stochastic_maxwell_ito_formula_formula}
	\|X(t)&\|^2_{L^{2}(D)^6}- \|X_0\|^2_{L^{2}(D)^6}\notag\\
	&=\int_{0}^{t}2\re\langle X(s),Y(s)\rangle_{L^2(D)^6}+\|Z(s)\|_{L^{2}(U;L^{2}(D)^6}^2\operatorname{ds}+2\int_{0}^{t}\re\big\langle X(s),Z(s)dW(s)\big\rangle_{L^2(D)^6}.
	\end{align}
	holds almost surely for all $ t\in [0,T] $ and $ X\in L^{2}(\Omega;C(0,T;L^{2}(D)))^6. $
\end{Lemma}
\begin{proof}
We plug in $ \phi=S_n\Phi $ for $ \Phi\in C_c^{\infty}(D)^6 $. Note, that by Lemma \ref{stochastic_maxwell_commutation_of_M_and_S_n}, $ S_n $ and $ M $ commute. Moreover, $ R(S_n)\subset D(M). $
 Consequently, since $ S_n $ is self-adjoint and $ \Phi $ is chosen arbitrarily, we obtain 
	\begin{align*}\label{stochastic_maxwell_ito_formula_convultion}
	S_n X(t)-S_n X_0=\int_{0}^{t}MS_nX(s)+S_nY(s)\operatorname{ds}+\int_{0}^{t}S_nZ(s) dW(s).
	\end{align*} 
	almost surely for all $ t\in [0,T] $ in $ L^{2}(D)^6 $ and we can apply the standard It\^o formula for Hilbert space valued processes (see e.g. \cite{da_prato_zabczyk_stochastic_equations_in_infinite_dimensions}, Theorem $ 4.32 $) to get
\begin{align*}
\|S_n X(t)\|^2_{L^{2}(D)^6}&- \|S_n X_0\|^2_{L^{2}(D)^6}\\
=&\int_{0}^{t}2\re\langle S_nX(s),M S_nX(s)\rangle_{L^2(D)^6} +2\re\langle S_nX(s),S_nY(s)\rangle_{L^2(D)^6}\\
&+\|S_n Z(s)\|_{L^{2}(U;L^{2}(D))^6}^2\operatorname{ds}
+2\int_{0}^{t}\re\big\langle S_nX(s),S_nZ(s)dW(s)\big\rangle_{L^2(D)^6}.
\end{align*}
By Lemma \ref{stochastic_maxwell_maxwell_skew} and Proposition \ref{stochastic_maxwell_spectral_multiplier_l2_properties} $ S_nMS_n $ is skew-adjoint and the first term on the right hand side drops. In all the other terms, we can pass to the limit. Thereby we need, that $ S_n u\to u $ for $ n\to\infty $ in $ L^{q+2}(D)^6 $ and $ L^{\frac{q+2}{q+1}}(D)^6 $ (see Lemma \ref{stochastic_maxwell_properties_galerkin_uniformly_in_n}). This finally yields
\begin{align*}
\|X(t)&\|^2_{L^{2}(D)^6}- \| X_0\|^2_{L^{2}(D)^6}\\
&=\int_{0}^{t} 2\re\langle X(s),Y(s)\rangle_{L^2(D)^6}+\|Z(s)\|_{L^{2}(U;L^{2}(D)^6)}^2\operatorname{ds}
+2\int_{0}^{t}\re\big\langle X(s),Z(s)dW(s)\big\rangle_{L^2(D)^6}.
\end{align*}
This identity together with $ X\in L^{q+2}(\Omega\times[0,T]\times D)^6\cap L^{2}(\Omega\times [0,T]\times D)^6 $ implies $ u\in L^{2}(\Omega;L^{\infty}(0,T;L^{2}(D)))^6. $ The pathwise continuity in time can be shown by applying the above result to the difference $ X(t)-X(s). $ This closes the proof.
\end{proof}

At first, we assume $ G\equiv 0 $ and solve $ (\operatorname{WSEE}) $ without retarded material law. The reason for this simplification is that we make use of the monotone structure of the rest of the equation. We start with a Galerkin approximation with the spectral projection $ P_n, $ we defined in section $ 2. $
We investigate the truncated equation
\begin{equation}\label{maxwell_nonlinear_galerkin_approx_equation}
\begin{cases}
du_n&=\big[P_nMu_n-P_nF(u_n(t))+P_nJ\big]\operatorname{dt}+P_nB(t,u_n(t))dW(t),\\
u_n(0)&=P_nu_0
\end{cases}
\end{equation}
in the range of $ P_n $. This is a stochastic ordinary differential equation in $ R(P_n)\subset L^{2}(D)^6 $ with a locally Lipschitz nonlinearity (see Lemma \ref{stochastic_maxwell_nonlinearity_properties}). Hence, there exists an increasing sequence of stopping times $ (\tau^{(m)}_n)_{m\in\N} $ with $  0<\tau^{(m)}_n\leq T $ almost surely, a stopping time $ \tau_n=\lim_{m\to\infty}\tau_n^{(m)} $ and an adapted process $ u_n:\Omega\times [0,\tau)\to R(P_n) $ with continuous paths, that solves \eqref{maxwell_nonlinear_galerkin_approx_equation}. Moreover, we have the blow-up alternative
\begin{align}\label{maxwell_nonlinear_galerkin_approx_equation_blow_up}
P\Big\{\tau_n<T, \sup_{t\in[0,\tau)}\|u_n(t)\|_{L^{2}(D)^6}<\infty\Big \}=0.
\end{align}
The next result shows $ \tau_n=T $ for every $ n\in\N $ and a uniform estimate for $ u_n. $
\begin{Proposition}\label{stochastic_maxwell_weak_approx_equation_solve}
	We have $ \tau_n=T $ for every $ n\in\N $ and $ u_n $ additionally satisfies
	\begin{align*}
	\sup_{n\in\N}\E \sup\limits_{t\in[0,T]}\|u_n(t)\|^2_{L^{2}(D)^6}+\sup_{n\in\N}\E\int_{0}^{T}\int_{D}|u_n(t,x)|^{q+2}\operatorname{dx}\operatorname{dt}<\infty.
	\end{align*}
\end{Proposition}

\begin{proof}
Lemma \ref{stochastic_maxwell_ito_formula} applied to $ u_n $, the self-adjointness of $ P_n $ and $ P_n^2=P_n2 $ yield
	\begin{align*}
	\|u_n(s)&\|_{L^{2}(D)^6}^{2}-\|P_nu_0\|_{L^{2}(D)^6}^{2}\\
	=&2\int_{0}^{s}\re\langle u_n(r),-|u_n(r)|^qu_n(r)+J(r)\rangle_{L^2(D)^6}\operatorname{dr}\\
	&+2\int_{0}^{s}\re\big\langle u_n(r),B(s,u_n(r))dW(r)\big\rangle_{L^2(D)^6}+\int_{0}^{s}\|P_nB(r,u_n(r))\|_{L^{2}(U;L^{2}(D)^6)}^{2}\operatorname{dr}.
	\end{align*}
	almost surely for every $ s\in[0,\tau_n^{(m)}].$  This expression simplifies to
	\begin{align}\label{stochastic_maxwell_transformation_approx_equation_solve_bla_bla_bla}
	\|u_n(s)&\|_{L^{2}(D)^6}^{2}+\int_{0}^{s}\int_{D}|u_n(s,x)|^{q+2}\operatorname{dx}\operatorname{dt}-\|P_nu_0\|_{L^{2}(D)^6}^{2}\notag\\
	\leq&2\int_{0}^{s}\re\big\langle u_n(r),J(r)\rangle_{L^2(D)^6}+\tfrac{1}{2}\|B(r,u_n(r))\|_{L^{2}(U;L^{2}(D)^6)}^{2}\operatorname{dr}\notag\\
	&+2\sum_{j=1}^{N}\int_{0}^{s}\re\big\langle u_n(r),B(s,u_n(r))dW(r)\big\rangle_{L^2(D)^6}.
	\end{align}
	almost surely for every $ s\in[0,\tau_n^{(m)}].  $ Since the second term on the left hand side is positive, we can drop it for a moment. We first take the supremum over $ [0,\tau_n^{(m)}\wedge t] $ for $ t\in [0,T] $ and than the expectation value and estimate the remaining quantities term by term. We start with the deterministic part using $ [\operatorname{W6}] $ and $ [\operatorname{W7}]. $
	\begin{align*}
	\E\sup_{s\in[0,\tau_n^{(m)}\wedge t]}\Big|\int_{0}^{s}\re\big\langle& u_n(r),J(r)\rangle_{L^2(D)^6}+\tfrac{1}{2}\|B(r,u_n(r))\|_{L^{2}(U;L^{2}(D)^6)}^{2}\operatorname{dr}\Big|\\
	&\lesssim_{B} \int_{0}^{t}\E \ind_{s\leq \tau_n^{(m)}}\|u_n(s)\|_{L^{2}(D)^6}\|J(s)\|_{L^{2}(D)^6}+\|u_n(s)\|_{L^{2}(D)^6}^2\operatorname{ds}\\
	&\lesssim \int _{0}^{t}\E\sup_{r\in [0,s\wedge\tau_n^{(m)}]}\|u_n(r)\|_{L^{2}(D)^6}^{2}\operatorname{ds}+\|J\|_{L^{2}(\Omega\times[0,T]\times D)^6}^{2}.
	\end{align*}
	The stochastic part can be estimated with the Burkholder-Davies-Gundy inequility. 
	\begin{align*}
	\E\sup_{s\in[0,t\wedge\tau_n^{(m)}]}\Big|\int_{0}^{s}&\re\langle u_n(s),B(s,u_n(s))\rangle_{L^2(D)^6}dW(s)\Big|\\
	&\leq C\E\Big( \int_{0}^{\tau_n^{(m)}\wedge t}\big|\langle u_n(s),B(s,u_n(s)\rangle_{L^{2}(U,L^{2}(D)^6)}\big|^2\operatorname{ds}\Big)^{1/2}\\
		&\leq \widetilde{C}\E\Big(\sup_{s\in[0,t\wedge \tau_n^{(m)}]}\|u_n(s)\|_{L^{2}(D)^6}\big(\int_{0}^{t\wedge \tau_n^{(m)}}\|u(t)\|_{L^{2}(D)^6}^2\operatorname{dt}\big)^{\frac{1}{2}}\Big)\\
		&\leq \frac{1}{4}\E\sup_{s\in[0,t\wedge \tau_n^{(m)}]}\|u_n(s)\|_{L^{2}(D)}^{2}+\widetilde{C}^2\int_{0}^{t}\E\sup_{r\in[0,s\wedge \tau_n^{(m)}]}\|u(r)\|_{L^{2}(D)^6}^2\operatorname{ds}
		\end{align*}
	Putting these estimates together, we get
	\begin{align*}
	\E&\sup_{s\in[0,t\wedge\tau_n^{(m)}]}\|u_n(s)\|_{L^{2}(D)^6}^{2}
	\lesssim \|u_0\|_{L^{2}(D)^6}^{2}+\|J\|_{L^{2}(\Omega\times[0,T]\times D)^6}^{2}+\int _{0}^{t}\E\sup_{r\in [0,s\wedge\tau_n^{(m)}]}\|u_n(r)\|_{L^{2}(D)^6}^{2}\operatorname{ds}
	\end{align*}
	for all $ t\in [0,T]. $	Consequently, Gronwall yields  
	\begin{align*}
	\E\sup_{t\in[0,\tau_n^{(m)}]}\|u_n(t)\|_{L^{2}(D)^6}^{2}\lesssim  \|J\|_{L^{2}(\Omega\times[0,T]\times D)}^{2}+\|u_0\|_{L^{2}(D)^6}^{2}.
	\end{align*}
	Now, we can go back to \eqref{stochastic_maxwell_transformation_approx_equation_solve_bla_bla_bla} and deal with the term, we dropped at first. The estimates of $ \E\sup_{t\in[0,\tau_n^{(m)}]}\|u_n(t)\|_{L^{2}(D)}^{2} $ imply
	\begin{align*}
	\E\int_{0}^{\tau_n^{(m)}}\int_{D}&|u_n(s,x)|^{q+2}\operatorname{dx}\operatorname{dt}\lesssim \|J\|_{L^{2}(\Omega\times[0,T]\times D)^6}^{2}+\|u_0\|_{L^{2}(D)^6}^{2}
	\end{align*}
	We use Fatou's Lemma to pass to the limit $ m\to\infty $ in these estimates. Note, that one can interchange $ \operatorname{sup} $ and $ \liminf $ in an upper estimate, since $ \liminf $ can be written in the form $ \sup \inf $ and supremums can be interchanged, whereas $ \sup \inf\leq\inf \sup $. Hence, we have
		\begin{align}\label{stochastic_maxwell_weak_solution_a_priori_esti}
		\E\sup_{t\in[0,\tau_n)}\|u_n(t)\|_{L^{2}(D)^6}^{2}+\E\int_{0}^{\tau_n}\int_{D}|u_n(s,x)|^{q+2}\operatorname{dx}\operatorname{dt}\lesssim \|J\|_{L^{2}(\Omega\times[0,T]\times D)^6}^{2}+\|u_0\|_{L^{2}(D)^6}^{2}.
		\end{align}
		 Consequently, we also have $ \tau_n=T $ almost surely. Indeed, there exists $ N\subset\Omega $ with $ \PP(N)=0 $, such that $ \Omega\setminus \big(N\cup \{\tau_n=T \} \big) $ can be decomposed into disjoint sets 
		\begin{align*}
		\Big\{\tau_n <T, \sup_{t\in[0,\tau_n)}\|u_n(t)\|^2_{L^2(D)^6}<\infty\Big \},\ \Big\{\tau_n<T, \sup_{t\in[0,\tau_n)}\|u_n(t)\|^2_{L^2(D)^6}=\infty\Big \}.
		\end{align*}
		The first set has measure zero by \eqref{maxwell_nonlinear_galerkin_approx_equation_blow_up} and the second one has measure zero, since \eqref{stochastic_maxwell_weak_solution_a_priori_esti} implies $ \sup_{t\in[0,\tau_n)}\|u_n(t)\|_{L^2(D)^6}<\infty $ almost surely. Pathwise uniform continuity on $ [0,T] $ follows from Lemma \ref{stochastic_maxwell_ito_formula}. This closes the proof.
\end{proof}
In Proposition \ref{stochastic_maxwell_weak_approx_equation_solve}, we derived uniform estimates for $ u_n$. As a consequence, Lemma \ref{stochastic_maxwell_properties_nonlinearity_monotonicity} yields the uniform boundedness of $ F(u_n) $ in $ L^{\frac{q+2}{q+1}}(\Omega\times[0,T]\times D)^6 $. Thus, by Banach-Alaoglu, there exists $ u\in L^{2}(\Omega;L^{\infty}(0,T;L^{2}(D)^6))  $, $ N\in L^{\frac{q+2}{q+1}}(\Omega\times[0,T]\times D)^6 $, $ \widetilde{B}\in L^{2}(\Omega\times[0,T];L^{2}(U;L^{2}(D)))^6 $ and subsequences, still indexed with $ n $, such that
\begin{itemize}
	\item[a)] $ u_n\to u $ for $ n\to\infty $ in the $ \operatorname{weak}^{*} $ sense in $ L^{2}(\Omega;L^{\infty}(0,T;L^{2}(D)))^6. $
	\item[b)] $ u_n\to u $ for $ n\to\infty $ in the weak sense in $ L^{q+2}(\Omega\times[0,T]\times D)^6. $
	\item [c)] $ F(u_n)\to N $ for $ n\to\infty $ in the weak sense in $ L^{\frac{q+2}{q+1}}(\Omega\times[0,T]\times D)^6. $
	\item [d)] $ B(\cdot,u_n)\to \widetilde{B} $ for $ n\to\infty $ in the weak sense in $ L^{2}(\Omega\times[0,T];L^{2}(U;L^{2}(D)))^6. $
\end{itemize}
Testing \eqref{maxwell_nonlinear_galerkin_approx_equation} with $ \rho\phi $ for arbitrary $ \rho\in L^{q+2}(\Omega\times [0,T]) $ and $ \phi\in \bigcup_{n=1}^{\infty}R(P_n) $, the symmetry of $ P_n $ and the skew-symmetry of $ M $ yield
\begin{align*}
\E\int_{0}^{T}\langle u_n(t)-u_0,&\phi\rangle_{L^2(D)^6}\rho(t)\operatorname{dt}\\
=&\E\int_{0}^{T}\int_{0}^{t}-\langle u_n(s),MP_n\phi\rangle_{L^2(D)^6}+\langle-F(u_n(s))+J(s),P_n\phi\rangle_{L^2(D)^6}\operatorname{ds}\rho(t)\operatorname{dt}\\
&+\E\int_{0}^{T} \int_{0}^{t}\langle B(s,u_n(s)),P_n\phi\rangle_{L^2(D)^6}dW(s)\rho(t)\operatorname{dt}.
\end{align*}
By weak convergence, we can pass to the limit and obtain
\begin{align*}
\E\int_{0}^{T}\langle u(t)-P_nu_0,&\phi\rangle_{L^2(D)^6}\rho(t)\operatorname{dt}\\
=&\E\int_{0}^{T}\int_{0}^{t} -\langle u(s),M\phi\rangle_{L^2(D)^6}+\langle y(s),-N(s)+J(s),\phi\rangle_{L^2(D)^6}\operatorname{ds}\rho(t)\operatorname{dt}\\
&+\E\int_{0}^{T} \int_{0}^{t}\langle \widetilde{B}(s),\phi\rangle_{L^2(D)^6}dW(s)\rho(t)\operatorname{dt}.
\end{align*}
Thereby, we used $ P_n\phi=\phi $ for $ n $ large enough since $ \phi\in \bigcup_{n=1}^{\infty}R(P_n) $ and that linear and bounded operators are also weakly continuous. Since $ \rho $ was chosen arbitrarily, we finally get
\begin{align}\label{stochastic_maxwell_weak_limit}
\langle u(t)-u_0,&\phi\rangle_{L^2(D)^6}\notag\\
=&\int_{0}^{t} -\langle u(s),M\phi\rangle_{L^2(D)^6}+\langle u(s),-N(s)+J(s),\phi\rangle_{L^2(D)^6}\operatorname{ds}+ \int_{0}^{t}\langle \widetilde{B}(s),\phi\rangle_{L^2(D)^6}dW(s).
\end{align}
Hence, by density (see Lemma \ref{stochastic_maxwell_density}), this holds true for every $ \phi\in D(M)\cap L^{q+2}(D)^6. $ To show that $ u $ is a weak solution of $ (\operatorname{WSEE}) $ with $ G\equiv 0, $ it remains to show $ N=F(u) $ and $ \widetilde{B}=B(\cdot,u). $ This will be done by adapting a standard argument for stochastic evolution equations with monotone nonlinearies (see \cite{prevot_rockner_a_concise_course_on_stochastic_pde}, proof of Theorem $ 4.2.4 $, page $ 86 $) to our situation. To do this, we just need an It\^o formula for $ \E e^{-Kt}\|u(t)\|_{L^{2}(D)^6}^2, $ although $ Mu(t)\notin L^{2}(D)^6. $ The rest follows the line of the original.
\begin{Lemma}\label{stochastic_maxwell_limit_ito_formula_regularization}
For any $ K>0, $ the It\^o formula
\begin{align*}
 \E &e^{-Kt}\|u(t)\|_{L^{2}(D)^6}^2-\E\|u_0\|_{L^{2}(D)^6}^2\\
 &=\E\int_{0}^{t} e^{-Ks}\langle u(s),-N(s)+J(s)\rangle_{L^2(D)^6}+e^{-Ks}\|\widetilde{B}(s)\|_{L^{2}(U;L^{2}(D)^6)}^2-Ke^{-Ks}\|u(s)\|_{L^{2}(D)^6}^2\operatorname{ds}
\end{align*}
holds true almost surely for all $ t\in [0,T]. $
\end{Lemma}
\begin{proof}
This formula is immediate by Lemma \ref{stochastic_maxwell_ito_formula}, the It\^o product rule and the fact, that the expectation of a stochastic integral is zero.
\end{proof}
All in all, we showed the following result.
\begin{Proposition}\label{stochastic_maxwell_gzero_unique_weak_solution}
If we assume $ [\operatorname{W1}]-[\operatorname{W5}] $, the equation $ (\operatorname{WSEE}) $ with $ G\equiv 0 $ has a unique weak solution $ u $ in the sense of Definition \ref{stochastic_maxwell_weak_solution_definition}. 
\end{Proposition}
Finally, we add a nontrivial the retarded material law $ G $ with a perturbation argument.
\begin{Theorem}\label{stochastic_maxwell_unique_weak_solution}
If we assume $ [\operatorname{W1}]-[\operatorname{W5}] $, the equation $ (\operatorname{WSEE}) $ has a unique weak solution $ u $ in the sense of Definition \ref{stochastic_maxwell_weak_solution_definition}. 
\end{Theorem}
\begin{proof}
Let $ T_0\in (0,T]. $ By Proposition \ref{stochastic_maxwell_gzero_unique_weak_solution}
\begin{equation*}
\begin{cases}
du(t)&=\big[Mu(t)-F(u(t))+(G\ast v)(t)+J(t)\big]\operatorname{dt}+B(t,u(t))dW_t,\\
u(0)&=u_0
\end{cases}
\end{equation*}
has for every $ v\in L^{2}(\Omega;C(0,T_0;L^{2}(D)))^6 $ a unique solution $ u=:Kv\in L^{2}(\Omega;C(0,T_0;L^{2}(D)^6)).  $ Indeed, by $ [\operatorname{W3}]$, 
$$  t\mapsto \int_{0}^{t}G(t-s)u(s)\operatorname{ds}\in L^{2}(\Omega\times [0,T]\times D)^6 $$ 
and thus $ G\ast v $ satisfies $ [\operatorname{W5}]. $ In the following, we will show that $ K $ is a contraction in $ X:= L^{2}(\Omega;C(0,T_0;L^{2}(D)))^6 $, if we choose $ T_0>0 $ small enough. For given $ v,w\in X, $ we calculate with Lemma \ref{stochastic_maxwell_ito_formula}
\begin{align*}
\|Kv(s)-Kw(s)\|_{L^{2}(D)^6}^2
=&\int_{0}^{s}2\re\langle Kv(r)-Kw(r),F(Kw(r))-F(Kv(r))+(G\ast (v-w))(r)\rangle_{L^2(D)^6}\\
&+ \|B(r,Kv(r))-B(r,Kw(r))\|_{L^{2}(U;L^{2}(D)^6)}^2\operatorname{dr}\\
&+2\int_{0}^{t}\re\big\langle Kv(r)-Kw(r),B(r,Kv(r))-B(r,Kw(r))dW(r)\big\rangle_{L^2(D)^6}.
\end{align*}
In the following estimates, we take the supremum over $ [0,t] $ for $ t\in [0,T_0] $ and afterwards the expectation value. We now estimate the occurring quantities term by term. 
\begin{align*}
&\int_{0}^{s}\re\langle Kv(r)-Kw(r),(G\ast (v-w))(r)\rangle_{L^2(D)^6}ds\\
&\leq \int_{0}^{s}\frac{1}{2}\|Kv(r)-Kw(r)\|_{L^{2}(D)^6}^2+\frac{1}{2}\big\|\int_{0}^{r}G(r-\lambda)(v(\lambda)-w(\lambda))\operatorname{d\lambda}\big\|_{L^{2}(D)^6}^2\operatorname{dr}\\
&\leq  \int_{0}^{s}\frac{1}{2}\sup_{\lambda\in [0,r]}\|Kv(\lambda)-Kw(\lambda)\|_{L^{2}(D)^6}^2\operatorname{dr}+\frac{T_0\|G\|^2_{L^{1}(0,T;B(L^{2}(D)^6))}}{2} \sup_{\lambda\in [0,T_0]}\|v(\lambda)-w(\lambda)\|_{L^{2}(D)^6}^2
\end{align*}
for all $ s\in [0,T_0]. $ We can drop the contribution of $ F, $ as
\[ \langle Kv(r)-Kw(r),F(Kw(s))-F(Kv(s))\rangle_{L^2(D)^6}\leq -\alpha\|Kv(r)-Kw(r)\|_{L^{q+2}(D)^6}^{q+2} \]
for all $ s\in[0,T_0] $ and some $ \alpha>0 $ by Lemma \ref{stochastic_maxwell_properties_nonlinearity_monotonicity}. Moreover, by $ [\operatorname{W4}], $ we have
\begin{align*}
\int_{0}^{t}\|B(s,Kv(s))-B(s,Kw(s))\|_{L^{2}(U;L^{2}(D)^6)}^2\operatorname{ds}\leq C^2\int_{0}^{t}\sup_{r\in[0,s]}\|Kv(r)-Kw(r)\|_{L^{2}(D)^6}^2\operatorname{ds}.
\end{align*}
Last but not least, the Burkholder-Davies-Gundy inequality and $ [\operatorname{W4}] $ yield
\begin{align*}
\E&\sup_{s\in [0,t]}\Big|\int_{0}^{s}\re\big\langle Kv(r)-Kw(r),B(r,Kv(r))-B(r,Kw(r))dW(r)\big\rangle_{L^2(D)^6}\Big|\\
&\leq C\E\Big(\int_{0}^{t}\big\|\langle Kv(r)-Kw(r),B(r,Kv(r))-B(r,Kw(r))\rangle_{L^{2}(D)^6)}\big\|_{L^{2}(U)}^2\operatorname{dr}\Big)^{1/2}\\
&\leq C\E\sup_{s\in [0,t]}\|Kv(s)-Kw(s)\|_{L^{2}(D)^6}\Big(\int_{0}^{t}\|B(r,Kv(r))-B(r,Kw(r))\|^2_{L^{2}(U;L^{2}(D)^6)}\operatorname{dr} \Big)^{1/2}\\
&\leq \frac{1}{4}\E\sup_{s\in [0,t]}\|Kv(s)-Kw(s)\|_{L^{2}(D)^6}^2+\widetilde{C}^2\int_{0}^{t}\E\sup_{r\in[0,s]}\|Kv(r)-Kw(r)\|^2_{L^{2}(D)^6}\operatorname{ds}.
\end{align*}
All in all, we derived
\begin{align*}
\E\sup_{s\in [0,t]}\|Kv(s)-Kw(s)\|_{L^{2}(D)^6}^2
\leq& \int_{0}^{t}2(1+2\widetilde{C}^2+C^2)\E\sup_{\lambda\in [0,r]}\|Kv(\lambda)-Kw(\lambda)\|_{L^{2}(D)^6}^2\operatorname{dr}\\
&+\frac{T_0\|G\|_{L^{\infty}(\Omega;L^{1}(0,T;B(L^{2}(D)^6)))}}{2} \E\sup_{\lambda\in [0,T_0]}\|v(\lambda)-w(\lambda)\|_{L^{2}(D)^6}^2
\end{align*}
for every $ t\in [0,T_0]. $ Hence, Gronwall implies
\begin{align*}
\E&\sup_{s\in [0,t]}\|Kv(s)-Kw(s)\|_{L^{2}(D)^6}^2\\
&\leq \frac{T_0\|G\|_{L^{\infty}(\Omega;L^{1}(0,T;B(L^{2}(D)^6)))}}{2} \Big(\E\sup_{\lambda\in [0,T_0]}\|v(\lambda)-w(\lambda)\|_{L^{2}(D)^6}^2 \Big)e^{2(1+2\widetilde{C}^2+C^2)T_0}.
\end{align*}
Now, we choose $ T_0>0 $ small enough to ensure that $ K $ is a contraction. In particular, by Banach's fixed point theorem, there exists $ u_1\in L^{2}(\Omega;C(0,T_0;L^{2}(D)^6)) $ solving $ \operatorname{(WSEE)} $ on $ [0,T_0] $ and from $ Ku_1=u_1, $ we deduce $ u_1\in L^{q+2}(\Omega\times[0,T_0]\times D)^6. $ Clearly, by continuity in time, we have $ u_1(T_0)\in L^{2}(\Omega\times D)^6 $ and $ \omega\mapsto u_1(\omega,T_0) $ is $ \mathcal{F}_{T_0} $-measurable.

Next, given $ v\in L^{2}(\Omega;C(T_0,2T_0;L^{2}(D)^6)) $, we consider the equation
\begin{equation*}
\begin{cases}
dy=\big[My-F(y)+\int_{0}^{T_0}G(\cdot-s)u_1(s)\operatorname{ds}+\int_{T_0}^{\cdot}G(\cdot-s)v(s)\operatorname{ds}+J\big]\operatorname{dt}+B(\cdot,y)dW_t,\\
y(T_0)=u_1(T_0) 
\end{cases}
\end{equation*}
for $ t\in[T_0,2T_0] $. By Proposition \ref{stochastic_maxwell_gzero_unique_weak_solution}, we have a unique solution $ y:=K_2v. $ This defines an operator $ K_2:L^{2}(\Omega;C(T_0,2T_0;L^{2}(D)))^6\to L^{2}(\Omega;C(T_0,2T_0;L^{2}(D)))^6. $ However, $ K_2v-Kw_2 $ can be estimated in the very same way as above, since the additional term $ \int_{0}^{T_0}G(\cdot-s)u_1(s)\operatorname{ds} $ vanishes in this difference. As a consequence, $ K_2 $ is a contraction on $ L^{2}(\Omega;C(T_0,2T_0;L^{2}(D)))^6 $ and has a unique fixed point $ u_2. $ Inductively, we construct, $ u_n\in L^{2}(\Omega;C((n-1)T_0,nT_0;L^{2}(D))^6 $  solving 
\begin{equation*}
\begin{cases}
dy=\big[My-F(y)+\int_{0}^{(n-1)T_0}G(\cdot-s)u_1(s)\operatorname{ds}+\int_{(n-1)T_0}^{\cdot}G(\cdot-s)y(s)\operatorname{ds}+J\big]\operatorname{dt}+B(\cdot,y)dW_t,\\
y((n-1)T_0)=u_{n-1}((n-1)T_0) 
\end{cases}
\end{equation*}
and stop, when $ nT_0\geq T. $ Finally, the process $ u:=\sum_{n=1}^{\lfloor \tfrac{T}{T_0}\rfloor+1}u_n\ind_{[(n-1)T_0,nT_0)} $ solves the $ (\operatorname{WSEE}) $ on $ [0,T] $ and satisfies
\[ u\in L^{2}(\Omega;C(0,T;L^{2}(D)))^6\cap L^{q+2}(\Omega\times[0,T]\times D)^6. \]
By construction, $ u $ is unique on every interval $ [(n-1)T_0,nT_0), $ which implies uniqueness on $ [0,T]. $
\end{proof}
    \section{Existence and uniqueness of a strong solution}
In this section, we will discuss the following stochastic Maxwell equation 
\begin{equation*}
	(\operatorname{MSEE})\begin{cases}
		du&=\big[Mu-|u|^{q}u+G\ast u+J\big]\operatorname{dt}+\sum_{n=1}^{N}\big[b_n+iB_nu\big] d\beta_n(t),\\
		u(0)&=u_0.
	\end{cases}
\end{equation*}
on $ L^{2}(D)^6 $ with a monotone polynomial nonlinearity and a retarded material law and we derive existence and uniqueness of a strong solution in the sense of partial differential equations. For sake of readability, we sometimes write $ F(u):=|u|^qu. $ Before we start, we explain our solution concept.
\begin{Definition}\label{stochastic_maxwell_defintion_strong_solution}
A weak solution $ u $ is called strong solution of $ (\operatorname{MSEE}) $ if it additionally satisfies
\[ Mu\in L^{2}(\Omega;L^{\infty}(0,T;L^{2}(D)))^6+L^{\frac{q+2}{q+1}}(\Omega\times[0,T]\times D )^6. \]
\end{Definition}
Note, that in case of a bounded domain $ D\subset\R^{3} $, this integrability property reduces $ Mu\in L^{\frac{q+2}{q+1}}(\Omega\times[0,T]\times D )^6. $ We make the following assumptions.
\begin{itemize}
		\item [\lbrack\text{M1}\rbrack] Let $ q\in (1,2] $ and $ D\subset\R^{3} $ be a bounded $ C^{1} $- domain or $ D=\R^{3} $.
	\item[\lbrack\text{M2}\rbrack] Let $ u_0$ be strongly $ \mathcal{F}_0 $-measurable with $$ \E\|Mu_0\|_{L^{2}(D)^6}^2+\E\|u_0\|^{2(q+1)}_{L^{2(q+1)}(D)^6}<\infty. $$
	\item[\lbrack\text{M3}\rbrack] Let $ G\in L^{\infty}(\Omega;W^{1,1}(0,T;B(L^{2}(D)^6))), $ such that $ \omega\mapsto G(t)x $ is for all $ x\in L^{2}(D)^6 $ and all $ t\in [0,T] $ strongly $ \mathcal{F}_t $-measurable. 
	\item[\lbrack\text{M4}\rbrack] Let $ J:L^{2}(\Omega; W^{1,2}(0,T;L^{2}(D)))^6 $ be $ \F $-adapted. 
	\item[\lbrack\text{M5}\rbrack] Let $ b_j:L^{2}(\Omega; W^{1,2}(0,T;L^{2}(D)))^6 $, $ j=1,\dots, N $ be $ \F $-adapted. If $ q\in (1,2), $ we additionally assume  $ b_j\in L^{\frac{2(q+2)}{2-q}}(\Omega\times [0,T]\times D)^6,  $ whereas we need $ b_j\in L^{\infty}(\Omega\times [0,T]\times D)^6 $, if $ q=2. $ Moreover, we assume 
	\[ P_n\big(b_je^{-i\sum_{l=1}^{N}B_l\beta_l}\big)=b_je^{-i\sum_{l=1}^{N}B_l\beta_l} \]
	for $ n\in\N $ large enough in the case $ q=2. $
	\item[\lbrack\text{M6}\rbrack] Let $ B_j\in W^{1,\infty}(D) $ for $ j=1,\dots,N. $
\end{itemize}

At first, we assume $ G\equiv 0 $ and solve $ (\operatorname{MSEE}) $ without retarded material law as in the last section. The reason for this simplification is that we make use of the monotone structure of the rest of the equation. As described in the introduction, we failed to derive an a priori estimate for $ Mu $ directly with It\^o's formula and Gronwall, since we could not control the terms $ \|F''(u)(B_ju(s),B_ju(s)\|_{L^{2}(D)^6} $ and $  \|F'(u)(B_ju(s)\|_{L^{2}(D)^6}.  $ Hence, we start with a rescaling transformation, such that the multiplicative noise vanishes. We end up with
\begin{equation*}
(\operatorname{TSEE})\begin{cases}
dy(t)&=[My(t)-|y(t)|^{q}y(t)+A(t)y(t)+\widetilde{J}(t)]\operatorname{dt}+\sum_{i=1}^{N}\widetilde{b}_i(t)\ d\beta_i(t),\\
u(0)&=u_0,
\end{cases}
\end{equation*}
where $ A(t) $, $ \tilde{J} $ and the new additive noise $ \sum_{j=1}^{N}\tilde{b}_j\ d\beta_j $ are given by 
\begin{align*}
A(t,x)y(t,x):&=\tfrac{1}{2}\sum_{j=1}^{N}B_j(x)^{2}y(t,x)+\sum_{j=1}^{N}i\beta_j(t)\binom{\nabla B_j(x)\times y_{2}}{-\nabla B_j(x)\times y_{1}},\\
\tilde{J}(t,x):&=\sum_{j=1}^{N}\big(-ib_j(t,x)B_j(x)+J(t,x)\big)e^{-i\sum_{n=1}^{N}B_n(x)\beta_n(t)},  \\
\tilde{b}_i(t,x):&=b_i(t,x)e^{-i\sum_{j=1}^{N}B_j(x)\beta_j(t)} 
\end{align*}
for $ t\in [0,T],$ $ x\in D $ and $ i=1,\dots,N. $ First, we show that a solution of $ (\operatorname{TSEE}) $ can be transformed to a solution of $ (\operatorname{MSEE}). $
\begin{Proposition}\label{stochastic_maxwell_transformation}
	A stochastic process $ u:\Omega\times [0,T]\to L^{2}(D) $ with almost surely continuous paths
	 is a strong solution of $ (\operatorname{MSEE}) $ with $ G\equiv 0 $ if and only if the process $ y(t):=\expbt u(t) $ has almost surely continuous paths, satisfies 
	\begin{itemize}
\item [i)] $ \E\sup_{t\in [0,T]}\|y(t)\|^2_{L^2(D)^6}+\E\int_{0}^{T}\int_{D}|y(t,x)|^{q+2}\operatorname{dx}\operatorname{dt}<\infty $
\item[ii)] $My+i\sum_{j=1}^{N}\beta_j\binom{\nabla B_j\times y_{2}}{-\nabla B_j\times y_{1}}\in L^{\frac{q+2}{q+1}}(\Omega\times[0,T]\times D)^6+L^{2}(\Omega;L^{\infty}(0,T;L^{2}(D)))^6 $
	\end{itemize}
	and solves the equation $ (\operatorname{TSEE}). $
	\begin{proof}
		We assume that $ u $ is a solution of $ (\operatorname{MSEE}) $ in the sense of Definition \ref{stochastic_maxwell_defintion_strong_solution} with the described regularity properties. At first, we calculate $ d(e^{i\sum_{j=1}^{N}B_j\beta_n(t) }) $ with It\^o's formula and obtain
		\begin{align*}
		e^{i\sum_{j=1}^{N}B_j\beta_j(t)}-1=\sum_{j=1}^{N}\int_{0}^{t}iB_j\expbb d \beta_n(s)-\tfrac{1}{2}\sum_{j=1}^{N}\int_{0}^{t}B_j^{2}\expbb ds.
		\end{align*}
		Therefore, It\^o's product rule yields
		\begin{align}\label{maxwell_transformation_calculation_1}
		\langle y&(t),x'\rangle_{L^2(D)^6}-\langle u_0,x'\rangle_{L^2(D)^6}=\langle u(t),\expbbt x'\rangle_{L^2(D)^6}-\langle u_0,x'\rangle_{L^2(D)^6}\notag\\
		=&\sum_{j=1}^{N}\int_{0}^{t}-\langle u(s),\tfrac{1}{2}B_j^{2}\expbb x'\rangle_{L^2(D)^6}+\langle b_j(s)+iB_ju(s),iB_j\expbb x'\rangle_{L^2(D)^6} \operatorname{ds}\notag\\
		&+\int_{0}^{t} \langle Mu(s)-|u(s)|^{q}u(s)+J(s),\expbb x'\rangle_{L^2(D)^6}\operatorname{ds}\notag\\
		&+\sum_{j=1}^{N}\int_{0}^{t}\langle u(s),iB_j\expbb x'\rangle_{L^2(D)^6} +\langle b_j(s)+iB_ju(s),\expbb x'\rangle_{L^2(D)^6} d \beta_n(s)
		\end{align}
		almost surely for every $ x'\in C_c^{\infty}(D) $ and for every $ t\in [0,T] $. As a consequence, we have
		\begin{align*}
		y(t)-u_0=&\int_{0}^{t}\expb M\big(\expbb y(s)\big)-|y(s)|^{q}y(s)\operatorname{ds}\\
		&+\int_{0}^{t} \expb J+\sum_{j=1}^{N} \tfrac{1}{2} B_j^2y(s)-ib_j(s)B_j(s)\expb\operatorname{ds}\\
		&+\sum_{n=1}^{N}\int_{0}^{t}b_n(s)\expb d \beta_n(s)
		\end{align*}
		almost surely for every $ t\in [0,T] $. Here, we used that $ u\in L^{q+2}(\Omega\times [0,T]\times D)^6 $ implies $ |y|^{q}y\in L^{q+2}(\Omega\times [0,T]\times D)^6 . $  
Since we want to derive an equation for $ y, $ we have to commute the exponential function with $ M. $ Therefore we compute
\begin{align*}
My(t)&=M(\expbt u(t))\\
&=\binom{\curl (\expbt u_2(t))}{-\curl (\expbt u_1(t))}\\
&=\sum_{j=1}^{N}-i\beta_j(t)\binom{(\nabla B_j)\expbt\times u_2(t)}{-(\nabla B_j)\expbt\times u_1(t)}+\binom{\expbt \curl(u_2(t))}{-\expbt \curl(u_1(t))}\\
&=\sum_{j=1}^{N}i\beta_j(t)\binom{-\nabla B_j\times y_2(t)}{\nabla B_j\times y_1(t)}+\expbt Mu(t).
\end{align*}
		Inserting this into \eqref{maxwell_transformation_calculation_1} finally proves that $ y $ solves $ (\operatorname{TSEE}). $ The other direction follows the same lines. 
	\end{proof}
\end{Proposition}
We solve $ (\operatorname{TSEE}) $ by a refined Galerkin approximation of the skew-adjoint operator $ M $. To do this, we truncated the equation with the spectral multipliers $ P_n $ and $ S_n $, we defined in section $ 3. $ We study 
\begin{equation}\label{maxwell_nonlinear_galerkin_approx_equation_2}
\begin{cases}
dy_n(t)&=[P_nMy_n(t)-P_nF(y_n(t))+P_nA(t)y_n(t)+P_n\widetilde{J}(t)]\operatorname{dt}+\sum_{i=1}^{N}S_{n-1}\widetilde{b}_i(t)\ d\beta_i(t),\\
y_n(0)&=S_{n-1}u_0.
\end{cases}
\end{equation}
In the next Proposition, we derive a priori estimates for the solution exploiting the structure of the equation. 

\begin{Proposition}\label{stochastic_maxwell_transformation_approx_equation_solve}
	The truncated equation \eqref{maxwell_nonlinear_galerkin_approx_equation_2} has for every $ n\in\N $ a unique, pathwise continuous solution $ y_n:\Omega\times[0,T]\to L^{2}(D)^6$, that additionally satisfies 
	\begin{align}\label{stochastic_maxwell_transformation_approx_equation_solve_estimate}
	\E \sup\limits_{t\in[0,T]}\|y_n(t)\|^2_{L^{2}(D)^6}+&\E \int_{0}^{T}\|y_n(t)\|_{L^{q+2}(D)^6}^{q+2}dt\notag\\
	&\leq C\Big(\|\widetilde{J}\|_{L^{2}(\Omega\times[0,T]\times D)}^{2}+\sum_{j=1}^{N}\|\widetilde{b_j}\|_{L^{2}(\Omega\times[0,T]\times D)}^{2}+\|u_0\|_{L^{2}(D)}^{2}\Big)
	\end{align}
	for some constant $ C>0 $ only depending on $\sup_{j=1,\dots,N}\|B_j\|_{L^{\infty}(D)}, $ but not on $ n\in\N. $ 
\end{Proposition}
\begin{proof}
	First, we define the stopping time 
	$$ \tau_m:=\inf\big \{t\in [0,T]:|\beta_i(t)|>m\text{ for some }i=1,\dots,N\big \} $$ 
	and solve the equation
	\begin{equation}\label{stochastic_maxwell_tsee_aprox_truncated}
	\begin{cases}
	dy_{n}^{(m)}(t)&=[P_nMy_{n}^{(m)}(t)-P_nF(y_{n}^{(m)}(t))+P_nA^{(m)}(t)y_{n}^{m}(t)+P_n\widetilde{J}(t)]\operatorname{dt}+\sum_{i=1}^{N}S_{n-1}\widetilde{b}_i(t)\ d\beta_i(t),\\
	u(0)&=S_{n-1}u_0,
	\end{cases}
	\end{equation}
	where the truncated linear operator $ A^{(m)} $ is given by
	$$ A^{(m)}(t)y(t):=\sum_{j=1}^{N}\Big(i\beta_j(t\wedge\tau_m)\binom{\nabla B_j\times y_2(t)}{-\nabla B_j\times y_1(t)}+B_j^{2}y(t).$$

By Lemma Lemma \ref{stochastic_maxwell_nonlinearity_properties} and \ref{stochastic_maxwell_properties_galerkin}, this an ordinary stochastic differential equation in the closed subspace $ R(P_n) \subset  L^2(D)^6 $ with locally Lipschitz nonlinearity. The stopping time $ \tau_m $ is necessary at this point, since it leads to $ L^{\infty} $-coefficients that are required to apply the classical results for stochastic ordinary differential equations.

There exists a stopping time $ \tau^{(m,n)} $, an increasing sequence of stopping times $ (\tau^{(m,n)}_k)_k $ with $ \tau^{(m,n)}_k\to\tau^{(m,n)} $ almost surely for $ k\to\infty $ and a process $ y_n^{(m)} $ with $$ y_n^{(m)}\in C(0,\tau^{(m,n)}_k;L^2(D)^6) $$ almost surely, such that $ y_n^{(m)} $ solves \eqref{stochastic_maxwell_tsee_aprox_truncated} on $ [0,\tau^{(m,n)}_k] $. Moreover, we have the blow-up alternative
\begin{align}\label{maxwell_nonlinear_galerkin_approx_equation_blow_up_2}
\PP\Big\{\tau^{(m,n)}<T, \sup_{t\in[0,\tau^{(m,n)})}\|y_n(t)\|_{L^2(D)^6}<\infty\Big \}=0.
\end{align}
For the a priori estimate, we use the It\^o formula from Lemma \ref{stochastic_maxwell_ito_formula} to get
	\begin{align*}
	\|y_n^{(m)}(t)&\|_{L^{2}(D)^6}^{2}-\|u_0\|_{L^{2}(D)^6}^{2}\\
	=&2\int_{0}^{t}\re\langle y_n^{(m)},-|y_n^{(m)}(s)|^qy_n^{(m)}(s)+A^{(m)}(s)y_n^{(m)}(s)+\widetilde{J}(s)\rangle_{L^2(D)^6}\operatorname{ds}\\
	&+2\sum_{j=1}^{N}\int_{0}^{t}\re\langle y_n^{(m)}(s),S_{n-1}\widetilde{b}_j(s)\rangle_{L^2(D)^6}d\beta_j(s)+\sum_{j=1}^{N}\int_{0}^{t}\|S_{n-1}\widetilde{b}_j(s)\|_{L^2(D)^6}^{2}\operatorname{ds}.
	\end{align*}
	Using the skew-symmetry of the cross-product, we calculate
	\begin{align*}
\big\langle y_n^{(m)},i\beta_j(t\wedge\tau_m)\binom{\nabla B_j\times y_{n,2}^{(m)}}{-\nabla B_j\times y_{n,1}^{(m)}}\big\rangle_{L^2(D)^6}&=-\big\langle i\beta_j(t\wedge\tau_m)y_n^{(m)},\binom{\nabla B_j\times y_{n,2}^{(m)}}{-\nabla B_j\times y_{n,1}^{(m)}}\big\rangle_{L^2(D)^6}\\
&=\big\langle i\beta_j(t\wedge\tau_m)\binom{\nabla B_j\times y_{n,2}^{(m)}}{-\nabla B_j\times y_{n,1}^{(m)}},y_n^{(m)}\big\rangle_{L^2(D)^6},
	\end{align*}
	which implies 
	$$ \re\big\langle y_n^{(m)}(s),i\beta_j(t\wedge\tau_m)\binom{\nabla B_j\times y_{n,2}^{(m)}}{-\nabla B_j\times y_{n,1}^{(m)}}\big\rangle_{L^2(D)^6}=0. $$ 
	Hence, the expression from above simplifies to
	\begin{align}\label{stochastic_maxwell_transformation_approx_equation_solve_bla_bla}
	\|y_n^{(m)}(t)&\|_{L^{2}(D)^6}^{2}+2\int_{0}^{t}\int_{D}|y_n^{(m)}(s,x)|^{q+2}\operatorname{dx}\operatorname{dt}\notag\\
	=&\|u_0\|_{L^{2}(D)^6}^{2}+2\int_{0}^{t}\re\big\langle y_n^{(m)}(s),\widetilde{J}(s)+\sum_{j=1}^{N}B_j^2y_n^{(m)}(s)\big\rangle_{L^{2}(D)^6}\operatorname{ds}\notag\\
	&+2\sum_{j=1}^{N}\int_{0}^{t}\re\langle y_n^{(m)}(s),S_{n-1}\widetilde{b}_j(s)\rangle_{L^2(D)^6}d\beta_j(s)+\sum_{j=1}^{N}\int_{0}^{t}\|S_{n-1}\widetilde{b}_j(s)\|_{L^2(D)^6}^{2}\operatorname{ds}
	\end{align}
	almost surely for $ t\in [0,\tau_k^{(m,n)}]. $ Since the second term on the left hand side is positive, we can drop it for a moment. We first take the supremum over time and then the expectation value and estimate the remaining quantities term by term. We start with the deterministic part.
	\begin{align*}
	\E&\sup_{s\in[0,t\wedge\tau_k^{(m,n)}]}\Big|\int_{0}^{s}\re\big\langle y_n^{(m)}(r),\widetilde{J}(r)+\sum_{j=1}^{N}B_j^2y_n^{(m)}(r)\big\rangle\operatorname{dr}+\sum_{j=1}^{N}\int_{0}^{s}\|S_{n-1}\widetilde{b}_j(r)\|_{L^2(D)^6}^{2}\operatorname{dr}\Big|\\
	\leq& \E\int_{0}^{t\wedge\tau_k^{(m,n)}} \|y_n^{(m)}(r)\|_{L^{2}(D)^6}\|\widetilde{J}(r)\|_{L^{2}(D)^6}+\sup_{j=1,\dots,N}\|B_j\|^2_{L^{\infty}(D)}\|y_n^{(m)}(r)\|_{L^{2}(D)^6}^2\operatorname{dr}\\
	&+\sum_{j=1}^{N}\|\widetilde{b_j}\|_{L^{2}(\Omega\times[0,T]\times D)^6}^{2}\\
	\leq& \int _{0}^{t}\E\sup_{r\in [0,s\wedge\tau_k^{(m,n)}]]}\|y_n^{(m)}(r)\|_{L^{2}(D)^{6}}^{2}\big(\sup_{j=1,\dots,N}\|B_j\|^2_{L^{\infty}(D)}+\tfrac{1}{2}\big)\operatorname{ds}+\frac{1}{2}\|\widetilde{J}\|_{L^{2}(\Omega\times[0,T]\times D)^6}^{2}\\
	&+\sum_{j=1}^{N}\|\widetilde{b_j}\|_{L^{2}(\Omega\times[0,T]\times D)^6}^{2}.
	\end{align*}
	The stochastic part can be estimated with the Burgholder-Davies-Gundy inequility. 
	\begin{align*}
	\E\sup_{s\in[0,t\wedge\tau_k^{(m,n)}]}&\Big|\sum_{j=1}^{N}\int_{0}^{s}\re\langle y_n^{(m)}(s),S_{n-1}\widetilde{b}_j(s)\rangle_{L^2(D)^6}d\beta_j(s)\Big|\\
	&\leq C\E\Big(\sum_{j=1}^{N} \int_{0}^{t\wedge\tau_k^{(m,n)}}\big|\re\langle y_n^{(m)}(s),S_{n-1}\widetilde{b}_j(s)\rangle_{L^2(D)^6}\big|^{2}\operatorname{ds}\Big)^{1/2}\\
	&\leq C\E\sup_{s\in[0,t\wedge\tau_k^{(m,n)}]}\|y_n^{(m)}(s)\|_{L^{2}(D)^6}\Big(\sum_{j=1}^{N}\|S_{n-1}\widetilde{b}_j\|_{L^{2}([0,T]\times D)}^{2}\Big)^{1/2}\\
	&\leq \frac{1}{4}\E\sup_{s\in[0,t\wedge\tau_k^{(m,n)}]}\|y_n^{(m)}(s)\|_{L^{2}(D)^6}^{2}+C^2\E\sum_{j=1}^{N}\|\widetilde{b}_j\|_{L^{2}(\Omega\times[0,T]\times D)}^2.
	\end{align*}
	Putting these estimates together, we get
	\begin{align*}
	\E\sup_{s\in[0,t\wedge\tau_k^{(m,n)}]}&\|y_n^{(m)}(s)\|_{L^{2}(\Omega\times D)}^{2}\\
	\lesssim& \|u_0\|_{L^2(D)^6}^{2}+\|\widetilde{J}\|_{L^{2}(\Omega\times[0,T]\times D)}^{2}+\sum_{j=1}^{N}\|\widetilde{b_j}\|_{L^{2}(\Omega\times[0,T]\times D)}^{2}\\
	&+\big(\sup_{j=1,\dots,N}\|B_j\|^2_{L^{\infty}(D)}+1\big)\int _{0}^{t}\E\sup_{r\in [0,s\wedge\tau_k^{(m,n)}]}\|y_n^{(m)}(r)\|_{L^{2}(D)^{6}}^{2}\operatorname{ds}
	\end{align*}
	Consequently, Gronwall yields
	\begin{align*}
	\E\sup_{s\in[0,t\wedge\tau_k^{(m,n)}]}&\|y_n^{(m)}(s)\|_{L^{2}(D)^6}^{2}\\
	&\lesssim _{B_j} \Big(\|\widetilde{J}\|_{L^{2}(\Omega\times[0,T]\times D)^6}^{2}+\sum_{j=1}^{N}\|\widetilde{b_j}\|_{L^{2}(\Omega\times[0,T]\times D)^6}^{2}+\|u_0\|_{L^{2}(\Omega\times D)}^{2}\Big)
	\end{align*}
	for every $ t\in [0,T] $. Next, we pass to the limit $ k\to\infty $ with Fatou's Lemma. 
	\begin{align}\label{stochastic_maxwell_tsee_aprox_trunc_estimate1}
\E\sup_{t\in[0,\tau^{(m,n)})}\|y_n^{(m)}(t)\|_{L^{2}(D)^6}^{2}&\leq \liminf_{k\to\infty} \E\sup_{t\in[0,\tau_k^{(m,n)}]}\|y_n^{(m)}(t)\|_{L^{2}(D)^6}^{2}\notag\\
&\lesssim _{B_j} \Big(\|\widetilde{J}\|_{L^{2}(\Omega\times[0,T]\times D)}^{2}+\sum_{j=1}^{N}\|\widetilde{b_j}\|_{L^{2}(\Omega\times[0,T]\times D)}^{2}+\|u_0\|_{L^{2}(\Omega\times D)}^{2}\Big).
	\end{align}
Note that this bound is independent of $ m $ and $ n. $ In particular, this estimate implies $ \tau^{(m,n)}=T $ almost surely. Indeed, there exists $ N\subset\Omega $ with $ \PP(N)=0 $, such that $ \Omega\setminus \big(N\cup \{\tau^{(m,n)}=T \} \big) $ can be decomposed into disjoint sets 
\begin{align*}
\Big\{\tau^{(m,n)}<T, \sup_{t\in[0,\tau^{(m,n)})}\|y_n^{(m)}(t)\|_{L^2(D)^6}<\infty\Big \},\ \Big\{\tau^{(m,n)}<T, \sup_{t\in[0,\tau^{(m,n)})}\|y_n^{(m)}(t)\|_{L^2(D)^6}=\infty\Big \}.
\end{align*}
Here, the first set has measure zero by \eqref{maxwell_nonlinear_galerkin_approx_equation_blow_up_2} and the second one has measure zero, since \eqref{stochastic_maxwell_tsee_aprox_trunc_estimate1} implies $ \sup_{t\in[0,\tau^{(m,n)})}\|y_n^{(m)}(t)\|_{L^2(D)^6}<\infty $ almost surely.
As a consequence of \eqref{stochastic_maxwell_transformation_approx_equation_solve_bla_bla}, we also get 
\begin{align}\label{stochastic_maxwell_tsee_aprox_trunc_estimate2}
\E\int_{0}^{T}\int_{D}&|y_n^{(m)}(s,x)|^{q+2}\operatorname{dx}\operatorname{dt}\lesssim _{B_j} \Big(\|\widetilde{J}\|_{L^{2}(\Omega\times[0,T]\times D)^6}^{2}+\sum_{j=1}^{N}\|\widetilde{b_j}\|_{L^{2}(\Omega\times[0,T]\times D)^6}^{2}+\|u_0\|_{L^{2}(D)}^{2}\Big).
\end{align}
We already know that $ y_n^{(m)} $ is almost surely continuous on $ [0,T) $ as a function with values in $ L^{2}(D)^6 $, the pathwise continuity up to $ y_n^{(m)}(T)$ follows from Lemma \ref{stochastic_maxwell_ito_formula}.
	
It remains to take the limit $ m\to\infty. $ By uniqueness, we have $ y_n^{(m)}(\omega,t)=y_n^{(k)}(\omega,t) $ for almost all $ \omega\in\Omega $, all $ t\in [0,\tau_m] $ and for every $ k\geq m. $ Moreover, for almost all $ \omega\in\Omega $, there exists $ m(\omega), $ such that $ \tau_{m(\omega)}(\omega)=T. $ Hence, we the limit
$ y_n=\lim_{m\to\infty}y_n^{(m)} $ is well-defined, adapted and satisfies \eqref{stochastic_maxwell_tsee_aprox_truncated}. Again using Fatou's Lemma yields analogous estimates to \eqref{stochastic_maxwell_tsee_aprox_trunc_estimate1} and \eqref{stochastic_maxwell_tsee_aprox_trunc_estimate2} for $ y_n. $ This closes the proof.
\end{proof}
To obtain strong solutions, we need an estimate for $ My_n $, uniformly in $ n\in\N. $ In a deterministic setting, one would try to control $ y_n' $ using the structure of the equation and then use the uniform estimates for $ y_n $ to find a bound for $ My_n $. However, solutions of stochastic differential equations are not differentiable in time and hence, we have to follow a different approach. We derive an a priori estimate for 
\[ \Big\| P_nMy_n(t)-P_nF(y_n(t))+P_n\sum_{j=1}^{N}B_j^{2}y_n(t)+P_ni\beta_j(t)\binom{\nabla B_j\times y_{n,2}(t)}{-\nabla B_j\times y_{n,1}(t)}+P_n\widetilde{J}(t)\Big\|_{L^{2}(D)^6}^2. \]
To do this, we have to show, that this quantity is an It\^o process. 
\begin{Lemma}\label{stochastic_maxwell_transformation_lambda_ito_process}
	The stochastic process 
	\[ \Lambda_n(t):= P_nMy_n(t)-P_nF(y_n(t))+P_n\sum_{j=1}^{N}B_j^{2}y_n(t)+P_ni\beta_j(t)\binom{\nabla B_j\times y_{n,2}(t)}{-\nabla B_j\times y_{n,1}(t)}+P_n\widetilde{J}(t) \]
	is an It\^o process with
	\begin{align*}
	d\Lambda_n (t)=& P_n\Big[M\Lambda_n-F'(y_n(t))(\Lambda_n(t))+\sum_{j=1}^{N}\Big(i\beta_j(t)\binom{\nabla B_j\times \Lambda_{n,2}(t)}{-\nabla B_j\times \Lambda_{n,1}(t)}+B_j^{2}\Lambda_n\Big)\\
	&-\frac{1}{2}\sum_{j=1}^{N}B_j^2\Big(\sum_{k=1}^{N}-ib_k(t)B_k+J(t)\Big)\expbt\\ &+\Big(\sum_{k=1}^{N}-i\partial_tb_k(t)B_k+\partial_tJ(t)\Big)\expbt-\frac{1}{2}\sum_{j=1}^{N}F''(y_n)(S_{n-1}\widetilde{b}_j,S_{n-1}\widetilde{b}_j)\Big]\operatorname{dt}\\
	&+\sum_{j=1}^{N}P_n\Big[MS_{n-1}\widetilde{b_j}-F'(y_n)(S_{n-1}\widetilde{b_j})+\sum_{k=1}^{N}i\beta_k(t)\binom{\nabla B_k\times S_{n-1}\widetilde{b}_{j,2}(t)}{-\nabla B_k\times S_{n-1}\widetilde{b}_{j,1}(t)}\\
	&+\sum_{k=1}^{N}B_k^{2}S_{n-1}\widetilde{b_j}+i\binom{\nabla B_j\times y_{n,2}(t)}{-\nabla B_j\times y_{n,1}(t)}-iB_j\Big(\sum_{k=1}^{N}-ib_k(t)B_k+J(t)\Big)\expbt \Big]d\beta_j
	\end{align*}
	
	almost surely for every $ t\in [0,T]. $ 
\end{Lemma}
\begin{proof}
	With Lemma \ref{stochastic_maxwell_nonlinearity_properties} and Lemma \ref{stochastic_maxwell_properties_galerkin}, one shows, that $ P_nF(y_n) $ is an It\^o process in $ L^{2}(D)^6 $
	with \[ d(P_nF(y_n))=\Big[P_nF'(y_n)\Lambda_n+\tfrac{1}{2}\sum_{j=1}^{N}P_nF''(y_n)(S_{n-1}\widetilde{b}_j,S_{n-1}\widetilde{b}_j)\Big]dt+ \sum_{j=1}^{N}P_nF'(y_n)S_{n-1}\widetilde{b}_jd\beta_j.\]
	 Moreover, by the product rule,  
	$$  P_n\widetilde{J}(t,x)=P_n\Big(\sum_{j=1}^{N}-ib_j(t,x)B_j(x)+J(t,x)\Big)\expbt $$
	is an It\^o process in $ L^2(D)^6 $ of the form 
	\begin{align*}
	d(P_n&\widetilde{J})(t)\\
	=&P_n\Big(-\frac{1}{2}\sum_{j=1}^{N}B_j^2\Big(\sum_{k=1}^{N}-ib_k(t)B_k+J(t)\Big)+\sum_{k=1}^{N}-i\partial_tb_k(t)B_k+\partial_tJ(t)\Big)\expbt\operatorname{dt}\\
	&-P_n\sum_{j=1}^{N}\Big[iB_j\big(\sum_{k=1}^{N}-ib_k(t)B_k+J(t)\big)\expbt \Big]d\beta_j 
	\end{align*}
	The remaining expression $ \Lambda_n+P_nF(y_n)-P_n\widetilde{J} $ is a $ C^{2} $-function of the It\^o processes 
	$$ dy_n(t,x)=\Lambda_n(t)dt+S_{n-1}\sum_{j=1}^{N}\widetilde{b}_jd\beta_j(t) $$ 
	and $ \beta_j $, $ j=1,\dots,N. $ Therefore, we can calculate $ d(\Lambda_n+P_nF(y_n)-P_n\widetilde{J}) $ with It\^o's formula. Thereby it is crucial that all occurring terms depend only linearly on $ y_n $ and $ \beta_j $ and consequently the second derivatives vanish. This finally proves the claimed result.
\end{proof}

\begin{Proposition}\label{stochastic_maxwell_transformation_approx_equation_energy}
	The process $ \Lambda_n $ satisfies the estimate
	\[ \E\sup_{t\in [0,T]}\|\Lambda_n(t)\|_{L^2(D)^6}^2\leq C \big(1+\E\|Mu_0\|_{L^2(D)^6}^{2}+\E\|u_0\|_{L^2(D)^6}^{2}+\E\|u_0\|^{2q+2}_{L^{2q+2}(D)^3}\big), \]
	with a constant $ C>0 $ depending on $ J,b_j $ and $ B_j $ for $ j=1,\dots,N $, but not on $ n\in\N. $ 
\end{Proposition}

\begin{proof}
	At first, we calculate $ \|\Lambda_n(t)\|^2_{L^2(D)^6} $ with the It\^o formula from Lemma \ref{stochastic_maxwell_ito_formula}.
	 We obtain
	\begin{align*}
	\|\Lambda_n&(t)\|^2_{L^{2}(D)^6}-\|\Lambda_n(0)\|^2_{L^{2}(D)^6}\\
	=&2\int_{0}^{t}\re\Big\langle\Lambda_n(s),-F'(y_n(s))(\Lambda_n(s))+\sum_{j=1}^{N}\Big(i\beta_j(s)\binom{\nabla B_j\times \Lambda_{n,2}(s)}{-\nabla B_j\times \Lambda_{n,1}(s)}+B_j^{2}\Lambda_n(s)\Big)\\
	&\ \ \ \ \ -\frac{1}{2}\sum_{j=1}^{N}B_j^2\Big(\sum_{k=1}^{N}-ib_k(s)B_k+J(s)\Big)\expb\\ &\ \ \ +\Big(\sum_{k=1}^{N}-i\partial_tb_k(s)B_k+\partial_tJ(s)\Big)\expbt\\
	&\ \ \ \ \ -\frac{1}{2}\sum_{j=1}^{N}F''(y_n)(S_{n-1}\widetilde{b}_j(s),S_{n-1}\widetilde{b}_j(s))\Big\rangle_{L^{2}(D)^6}\operatorname{ds}\\
	&\ \ \ \ \ +\int_{0}^{t}\Big\|MS_{n-1}\widetilde{b_j}(s)-F'(y_n)(S_{n-1}\widetilde{b_j}(s))+\sum_{k=1}^{N}i\beta_k(s)\binom{\nabla B_k\times S_{n-1}\widetilde{b}_{j,2}(s)}{-\nabla B_k\times S_{n-1}\widetilde{b}_{j,1}(s)}\\
	&\ \ \ \ \ +\sum_{k=1}^{N}B_k^{2}S_{n-1}\widetilde{b_j}(s)+i\binom{\nabla B_j\times y_{n,2}(s)}{-\nabla B_j\times y_{n,1}(s)}\\
	&\ \ \ \ \ -iB_j\Big(\sum_{k=1}^{N}-ib_k(s)B_k+J(s)\Big)\expb\Big\|_{L^2(D)^6}^2\operatorname{ds}\\
	&+2\sum_{j=1}^{N}\int_{0}^{t}\re\Big\langle\Lambda_n,MS_n\widetilde{b_j}(s)-F'(y_n)(S_{n-1}\widetilde{b_j}(s))+\sum_{k=1}^{N}i\beta_k(s)\binom{\nabla B_k\times S_{n-1}\widetilde{b}_{j,2}(s)}{-\nabla B_k\times S_{n-1}\widetilde{b}_{j,1}(s)}\\
	&\ \ \ \ \ +\sum_{k=1}^{N}B_k^{2}S_{n-1}\widetilde{b_j}(s)+i\binom{\nabla B_j\times y_{n,2}(s)}{-\nabla B_j\times y_{n,1}(s)}\\
	&\ \ \ \ \ -iB_j\Big(\sum_{k=1}^{N}-ib_k(s)B_k+J(s)\Big)\expb \rangle_{L^{2}(D)^6}d\beta_j(s).
	\end{align*} 
	As we have seen before in the proof of Proposition \ref{stochastic_maxwell_transformation_approx_equation_solve}, 
	the term $$ \re\big\langle \Lambda_n(s),\sum_{j=1}^{N}i\beta_j(s)\binom{\nabla B_j\times \Lambda_{n,2}(s)}{-\nabla B_j\times \Lambda_{n,1}(s)}\big\rangle_{L^{2}(D)^6} $$ vanishes. Moreover, by Lemma \ref{stochastic_maxwell_nonlinearity_properties}, we have $$ \re\langle \Lambda_n,F(y_n(s))'\Lambda_n(s)\rangle_{L^{2}(D)^6}\leq 0  $$
	almost surely for every $ s\in [0,T] $ and we can drop this term in an upper estimate because of the sign. We split this expression into the deterministic integral $ I_{\operatorname{det}} $ and the stochastic integral $ I_{\operatorname{stoch}}. $

	We take the supremum over time and afterwards the expectation value and we aim to control the left hand side with Gronwall. We start with an estimate for the deterministic integral $ I_{\operatorname{det}} $. Using Cauchy-Schwartz and the assumptions on $ B_j $, $ \nabla B_j $ $ \partial_tb_j, J $ and $ \partial_tJ $ from $ [\operatorname{M4}]-[\operatorname{M6}] $, we get
		\begin{align*}
\E\sup_{s\in [0,t]}|&I_{\operatorname{det}}(s)|\\
\lesssim & \int_{0}^{t}\|\Lambda_n(r)\|_{L^{2}(D)^6}^2+ \sum_{j=1}^{N}\|\Lambda_n(r)\|_{L^{2}(D)^6}\|F''(y_n)(S_{n-1}\widetilde{b}_j(r),S_{n-1}\widetilde{b}_j(r)\|_{L^{2}(D)^6}\\
&+\big\|MS_n\widetilde{b_j}(r)\|_{L^{2}(D)^6}^2+\|F'(y_n(r))(S_{n-1}\widetilde{b_j}(r))\|_{L^{2}(D)^6}^2		+\sum_{k=1}^{N}\|\beta_k(r)S_{n-1}\widetilde{b}_j(r)\|_{L^{2}(D)^6}^2\\
&+\sum_{k=1}^{N}\|S_{n-1}\widetilde{b_j}(r)\|_{L^{2}(D)}^2+\|y_n(r)\|_{L^{2}(D)}^2\operatorname{dr}
		\end{align*}
The growth estimates for $ F' $ and $ F'' $ from Lemma \ref{stochastic_maxwell_nonlinearity_properties} together with the uniform boundedness of $ S_{n-1} $ on $ L^{2}(D)^6 $ yield
			\begin{align*}
			\E\sup_{s\in [0,t]}|I_{\operatorname{det}}(s)|\lesssim & \int_{0}^{t}\|\Lambda_n(r)\|_{L^{2}(D)^6}^2+ \sum_{j=1}^{N}\||y_n|^{q-1}|S_{n-1}\widetilde{b}_j|^2\|_{L^{2}(D)^6}^2+\|M\widetilde{b_j}\|^2_{L^{2}(D)^6}+\|\widetilde{b_j}\|_{L^{2}(D)}^2\\
			&+\||y_n|^qS_{n-1}\widetilde{b_j}\|^2_{L^{2}(D)^6}+\sum_{k=1}^{N}\beta_k(r)^2\|\widetilde{b}_j(r)\|_{L^{2}(D)^6}^2+\|y_n(r)\|_{L^{2}(D)}^2\operatorname{dr}
			\end{align*}
In the following estimate, we have to distinguish the cases $ q\in (1,2) $ and $ q=2. $ We start with the first one. H\"older's inequality, the fact $ \beta_k\in L^{\alpha}(\Omega;C(0,T)) $ for every $ \alpha\in [2,\infty) $ and the boundedness of $ S_{n-1} $ on $ L^{p}(D)^6 $ for every $ p\in (1,\infty) $ with norm independent of $ n $ yield
	\begin{align*}
			\E\sup_{s\in[0,t]}|I_{\operatorname{det}}(t)|\lesssim & \int_{0}^{t}\E\sup_{r\in[0,s]}\|\Lambda_n(r)\|_{L^{2}(D)^6}^2\operatorname{ds}+ \|y_n\|_{L^{q+2}(\Omega\times[0,T]\times D)^6}^{2(q-1)}\|\widetilde{b}_j\|_{L^{\frac{4(q+2)}{4-q}}(\Omega\times[0,T]\times D)^6}^4\\
			&+\|M\widetilde{b_j}\|^2_{L^{2}(\Omega\times[0,T]\times D)^6}+\|y_n\|^{2q}_{L^{q+2}(\Omega\times[0,T]\times D)^6}\|\widetilde{b_j}\|_{L^{\frac{2(q+2)}{2-q}}(\Omega\times[0,T]\times D)^6}^2\\
				&+\|\widetilde{b}_j(s)\|_{L^{2+\varepsilon}(\Omega;L^{2}([0,T]\times D))^6}^2+\|\widetilde{b_j}\|_{L^{2}(\Omega\times [0,T]\times D)}^2+\|y_n(s)\|_{L^{2}(\Omega\times [0,T]\times D)}^2.
			\end{align*}
			for any $ \varepsilon >0. $	In the case $ q=2, $ the same argument yields
				\begin{align*}
				\E\sup_{s\in[0,t]}|I_{\operatorname{det}}(t)|\lesssim & \int_{0}^{t}\E\sup_{r\in[0,s]}\|\Lambda_n(r)\|_{L^{2}(D)^6}^2\operatorname{ds}+ \|y_n\|_{L^{4}(\Omega\times[0,T]\times D)^6}^{2}\|\widetilde{b}_j\|_{L^{8}(\Omega\times[0,T]\times D)^6}^4\\
				&+\|M\widetilde{b_j}\|^2_{L^{2}(\Omega\times[0,T]\times D)^6}+\|y_n\|^{4}_{L^{4}(\Omega\times[0,T]\times D)^6}\|S_{n-1}\widetilde{b_j}\|_{L^{\infty}(\Omega\times[0,T]\times D)^6}^2\\
				&+\|\widetilde{b}_j\|_{L^{2+\varepsilon}(\Omega;L^{2}([0,T]\times D))^6}^2+\|\widetilde{b_j}\|_{L^{2}(\Omega\times [0,T]\times D)}^2+\|y_n\|_{L^{2}(\Omega\times[0,T]\times D)}^2.
				\end{align*}
				for any $ \varepsilon>0. $ 
				At this point, we need the requirement $ S_{n-1}\widetilde{b}_j=\widetilde{b_j} $ for large enough $ n $ from $ [\operatorname{M5}]. $

Note, that we already bounded $ \|y_n\|_{L^{q+2}(\Omega\times[0,T]\times D)^6} $ and $ \|y_n\|_{L^{2}(\Omega\times[0,T]\times D)^6} $ in Proposition \ref{stochastic_maxwell_transformation_approx_equation_solve} uniformly in $ n $. Hence,we can conclude
\begin{align*}
\E\sup_{s\in[0,t]}|I_{\operatorname{det}}(s)|\lesssim &  1+\int_{0}^{t}\E\sup_{r\in[0,s]}\|\Lambda_n(r)\|_{L^{2}(D)^6}^2\operatorname{ds}
\end{align*}
and the estimate only depends on $ B_j,b_j $ and $ J $ but not on $ n\in\N. $ The stochastic term $ I_{\operatorname{stoch}} $ can be controlled in the same way as in the proof of Proposition \ref{stochastic_maxwell_transformation_approx_equation_solve} with the Burkholder-Davies-Gundy inequality and the assumptions on $ B_j,b_j $ and $ J $ together with the growth estimates for $ F' $ and $ F''. $ Thus, we end up with
\begin{align*}
\E\sup_{s\in[0,t]}\|\Lambda_n(s)\|_{L^{2}(D)}^2\lesssim &  \ 1+\E\|\Lambda_n(0)\|_{L^{2}(D)}^2+\int_{0}^{t}\E\sup_{r\in[0,s]}\|\Lambda_n(r)\|_{L^{2}(D)^6}^2\operatorname{ds}.
\end{align*}
	It remains to bound 
	$$ \Lambda_n(0)=P_nMS_nu_0-P_nF(S_nu_0)+P_n\sum_{j=1}^{N}B_j^2S_nu_0-P_n\sum_{j=1}^{N}ib_j(0)B_j+P_nJ(0) $$ 
	in $ L^{2}(\Omega\times D)^6 $ independent of $ n\in\N .$ Since both $ b_j$ and $ J $ are in $ L^{2}(\Omega;W^{1,2}(0,T;L^{2}(D)))^6, $ the corresponding initial data $ b_j(0) $ and $ J(0) $ is contained in $ L^{2}(\Omega\times D)^6. $ As a consequence, the uniform boundedness of $ S_n $ on $ L^{p}(D)^6 $ for every $ p\in (1,\infty) $ and of $ P_n $ on $ L^{2}(D)^6 $ yield
	\begin{align*}
	\E\sup_{s\in [0,T]}\|\Lambda_n(s)\|^2_{L^{2}(D)^6}\lesssim&\ 1+\E\|MS_nu_0\|^2_{L^{2}(D)^6}+\E\||S_nu_0|^qS_nu_0\|^2_{L^{2}(D)^6}\\
	&+\sum_{j=1}^{N}\|B_j\|_{L^{\infty}(D)}^2\|S_nu_0\|_{L^{2}(D)^6}^2\\
	\lesssim &\ 1+\E\|Mu_0\|^2_{L^{2}(D)^6}+\E\|u_0\|_{_{L^{2(q+1)}(D)^6}}^{2(q+1)}+\E\|u_0\|_{L^{2}(D)^6}^2.
	\end{align*}
	Finally, an application of Gronwall's Lemma closes the proof.
\end{proof}
In Proposition \ref{stochastic_maxwell_transformation_approx_equation_solve} and \ref{stochastic_maxwell_transformation_approx_equation_energy}, we derived uniform estimates for $ y_n$ and $ \Lambda_n. $ As a consequence, we also get the uniform boundedness of $ F(y_n) $, since
\begin{align*}
\|F(y_n)\|_{L^{\frac{q+2}{q+1}}(\Omega\times[0,T]\times D)}&\lesssim \||y_n|^{q+1}\|_{L^{\frac{q+2}{q+1}}(\Omega\times [0,T]\times D)^6}=\|y_n\|^{q+1}_{L^{q+2}(\Omega\times[0,T]\times D)^6}
\end{align*}
Hence, by Banach-Alaoglu, there exists processes $ y\in L^{2}(\Omega;L^{\infty}(0,T;L^2(D)))^6  $, $ N\in L^{\frac{q+2}{q+1}}(\Omega\times[0,T]\times D)^6 $, $ \Lambda\in L^{2}(\Omega;L^{\infty}(0,T;L^{2}(D)))^6 $ and subsequences, still indexed with $ n $, such that
\begin{itemize}
	\item[a)] $ y_n\to y $ for $ n\to\infty $ in the $ \operatorname{weak}^{*} $ sense in $ L^{2}(\Omega;L^{\infty}(0,T;L^2(D)))^6. $
	\item[b)] $ y_n\to y $ for $ n\to\infty $ in the weak sense in $ L^{2}(\Omega\times[0,T]\times D)^6. $
	\item [c)] $ F(y_n)\to N $ for $ n\to\infty $ in the weak sense in $ L^{\frac{q+2}{q+1}}(\Omega\times[0,T]\times D)^6. $
	\item[d)] $ \Lambda_n\to \Lambda $ for $ n\to\infty $ in the $ \operatorname{weak}$ sense in $ L^{2}(\Omega\times[0,T]\times D)^6.$
		\item[e)] $ \Lambda_n\to \Lambda $ for $ n\to\infty $ in the $ \operatorname{weak}^{*} $ sense in $ L^{2}(\Omega;L^{\infty}(0,T;L^2(D)))^6. $
\end{itemize}
In the next Lemma, we show that $ \Lambda $ has the correct form.
\begin{Lemma}\label{stochastic_maxwell_y_in_domain_M}
The process $ y:\Omega\times[0,T]\to L^{2}(D)^6 $ additionally satisfies $ y(\omega,t)\times \nu=0 $ on $ \partial D $ for almost all $ \omega\in\Omega $ and $ t\in [0,T] $. Moreover, we have
 $$ My+\sum_{j=1}^{N}i\beta_j\binom{\nabla B_j\times y_2}{-\nabla B_j\times y_1}\in L^{2}(\Omega;L^{\infty}(0,T;L^{2}(D)))^6+L^{\frac{q+2}{q+1}}(\Omega\times[0,T]\times D)^6, $$
and the identity
\begin{align*}
 \Lambda=My-N+\sum_{j=1}^{N}B_j^2y+i\beta_j\binom{\nabla B_j\times y_2}{-\nabla B_j\times y_1}+\widetilde{J}.
\end{align*}
holds.
\end{Lemma}
\begin{proof}
Let $ \phi:\Omega\times [0,T]\to \cup_{n=1}^{\infty}R(P_n) $ be a simple function. By weak convergence, we obtain
\begin{align*}
-\langle &y,M\phi\rangle_{L^{2}(\Omega\times[0,T]\times D)^6}\\
&=\lim_{n\to\infty}\langle y_n,M\phi\rangle_{L^{2}(\Omega\times[0,T]\times D)^6}\\
&=\lim_{n\to\infty}\big\langle \Lambda_n+P_nF(y_n)-P_n\widetilde{J}-P_n\sum_{j=1}^{N}B_jy_n-P_n\sum_{j=1}^{N}i\beta_j\binom{\nabla B_j\times y_{n,2}}{-\nabla B_j\times y_{n,1}},\phi\big\rangle_{L^{2}(\Omega\times[0,T]\times D)^6}\\
&=\big\langle \Lambda+N-\widetilde{J}-\sum_{j=1}^{N}B_jy-\sum_{j=1}^{N}i\beta_j\binom{\nabla B_j\times y_{2}}{-\nabla B_j\times y_{1}},\phi\big\rangle_{L^{2}(\Omega\times[0,T]\times D)^6}\\
\end{align*}
Here, we could drop the $ P_n, $ since $ P_n\phi=\phi $ for large enough $ n. $ By density of simple functions and by the density of $ \cup_{n=1}^{\infty}R(P_n) $ in $ D(M) $ and in $ L^{p}(D)^6 $ for every $ p\in (1,\infty) $ (see Corollary \ref{stochastic_maxwell_density}), we get
\[ -\langle y(t),M\psi\rangle_{L^{2}(D)^6}=\big\langle \Lambda(t)+N(t)-\widetilde{J}(t)-\sum_{j=1}^{N}B_jy(t)-\sum_{j=1}^{N}i\beta_j(t)\binom{\nabla B_j\times y_{2}(t)}{-\nabla B_j\times y_{1}(t)},\psi\big\rangle_{L^{2}( D)^6} \]
almost surely for almost every $ t\in [0,T] $ and for every  $ \psi\in D(M)\cap L^{q+2}(D)^6. $ By Lemma \ref{maxwell_stochastic_lemma_distributional_maxwell_operator}, this implies $ y_1(\omega,t)\times\nu =  0 $ on $ \partial D $ almost surely for almost every $ t\in [0,T] $ and 
\[ My=\Lambda+N-\widetilde{J}-\sum_{j=1}^{N}B_jy-\sum_{j=1}^{N}i\beta_j\binom{\nabla B_j\times y_{2}}{-\nabla B_j\times y_{1}}. \]
This identity also gives the claimed regularity result.
\end{proof}
 Consequently,we pass to the limit weakly in \eqref{maxwell_nonlinear_galerkin_approx_equation} and obtain
 \begin{equation}\label{maxwell_nonlinear_identify_N}
 \begin{cases}
 dy(t)&=[My(t)-N(t)+A(t)y(t)+\widetilde{J}(t)]\operatorname{dt}+\sum_{i=1}^{N}\widetilde{b}_i(t)\ d\beta_i(t),\\
 y_n(0)&=u_0.
 \end{cases}
 \end{equation}
 as an equation in $ L^{2}(\Omega;L^{\infty}(0,T;L^{2}(D)))^6. $ So far, we just showed $ y\in L^{2}(\Omega;L^\infty(0,T;L^{2}(D)^6)). $ However, Lemma \ref{stochastic_maxwell_ito_formula} implies pathwise continuity of $ t\mapsto y(t)\in L^{2}(D)^6 $.
 
 It remains to show $ N(t)=F(y(t)). $ But this proof is step by step the same as in Proposition \ref{stochastic_maxwell_gzero_unique_weak_solution} and uses the monotonicity of the deterministic part of the equation.
 
 All in all, we showed that $ y\in L^{q+2}(\Omega\times[0,T]\times D)^6\cap L^{2}(\Omega;C(0,T;L^{2}(D)^6)) $ solves
 \begin{equation}
 \begin{cases}
 dy(t)&=[My(t)-F(y(t))+A(t)y(t)+\widetilde{J}(t)]\operatorname{dt}+\sum_{i=1}^{N}\widetilde{b}_i(t)\ d\beta_i(t),\\
 y_n(0)&=u_0.
 \end{cases}
 \end{equation}
 as an equation in $ L^{2}(\Omega;L^{\infty}([0,T];L^{2}(D)))^6. $ Transforming the equation backwards with Proposition \ref{stochastic_maxwell_transformation}, we get the following result.
 
 \begin{Proposition}\label{stochastic_maxwell_solution_without_retarded}
$ (\operatorname{MSEE}) $ with $ G\equiv 0 $ has a unique strong solution $ u $ satisfying with
\begin{align*}
u\in L^{q+2}(\Omega\times[0,T]\times D)^6\cap L^{2}(\Omega;C(0,T;L^2(D)))^6
\end{align*}
and
\[ Mu\in L^{\frac{q+2}{q+1}}(\Omega\times[0,T]\times D)^6+L^{2}(\Omega;L^{\infty} (0,T;L^{2}(D)))^6.  \]
 \end{Proposition}
 \begin{proof}
 	The product rule yields $$ M(e^{i\sum_{j=1}^{N}B_j\beta_j}y)=e^{i\sum_{j=1}^{N}B_j\beta_j}My+ie^{i\sum_{j=1}^{N}B_j\beta_j}\sum_{j=1}^{N}\beta_j\binom{\nabla B\times y_2}{-\nabla B\times y_1} $$
 	and hence, we have 
 	$$ M(e^{i\sum_{j=1}^{N}B_j\beta_j}y)\in L^{\frac{q+2}{q+1}}(\Omega\times[0,T]\times D)^6+L^{2}(\Omega;L^{\infty} (0,T;L^{2}(D)))^6 $$ 
 	if and only if 
 	$$ My+i\sum_{j=1}^{N}\beta_j\binom{\nabla B\times y_2}{-\nabla B\times y_1}\in L^{\frac{q+2}{q+1}}(\Omega\times[0,T]\times D)^6+L^{2}(\Omega;L^{\infty} (0,T;L^{2}(D)))^6. $$
This holds true by \ref{maxwell_nonlinear_galerkin_approx_equation}. Consequently, we can apply Proposition \ref{stochastic_maxwell_transformation} and obtain a solution $ u $ of $ (\operatorname{MSEE}) $ with $ G\equiv 0. $
Uniqueness is immediate by Proposition \ref{stochastic_maxwell_gzero_unique_weak_solution}, since our solution is also a weak solution of the equation.
 \end{proof}
 Last but not least, we want to add the term $ (G\ast u). $ This leads to the main result of this article.
 \begin{Theorem}\label{stochastic_maxwell_main_result}
$ (\operatorname{MSEE}) $ has a unique solution $ u $ satisfying with
\begin{align*}
u\in L^{q+2}(\Omega\times[0,T]\times D)^6\cap L^{2}(\Omega;C(0,T;L^2(D)))^6
\end{align*}
and
\[ Mu\in L^{\frac{q+2}{q+1}}(\Omega\times[0,T]\times D)^6+ L^{2}(\Omega;L^{\infty}([0,T];L^{2}(D)))^6.   \]
 \end{Theorem}
 \begin{proof}
Let $ u\in L^{q+2}(\Omega\times[0,T]\times D)^6\cap L^{2}(\Omega;C(0,T;L^2(D)^6)) $ be the unique weak solution of $ (\operatorname{MSEE}) $ from Proposition \eqref{stochastic_maxwell_unique_weak_solution}. The expression $ (G\ast u)(t)=\int_{0}^{t}G(t-s)u(s)\operatorname{ds} $ is differentiable in time with
\[ \partial_t(G\ast u)(t)=G(0)u(t)+\int_{0}^{t}G'(t-s)u(s)\operatorname{ds}. \]
By $[\operatorname{M5}], $ both $ (G\ast u) $ and $ \partial_t(G\ast u) $ are contained in $ L^{2}(\Omega\times [0,T]\times D)^6. $ Hence, $ u $ is a solution of $ (\operatorname{MSEE}) $ with the current $ G\ast u+J  $ satisfying $ [\operatorname{M4}]. $ Consequently, $ u $ has the regularity properties from Proposition \ref{stochastic_maxwell_solution_without_retarded}. This closes the proof.
 \end{proof}
 
    \section{Remarks and discussion}
In this section, we want to compare our results to the literature and we discuss some instructive special cases of our assumptions.\\ \\
First, we want to mention, that Roach, Stratis and Yannacopoulus already treated our equation in the deterministic setting in \cite{roach_stratis_yannacopoulus}. They claim in Theorem $ 11.3.14 $, that 
\begin{equation*}
\begin{cases}
&u'(t)=\kappa^{-1}Mu(t)-\kappa^{-1}|u(t)|^qu(t)+\kappa^{-1}(G\ast u)(t)+\kappa^{-1}J(t),\\
&u(0)=u_0
\end{cases}
\end{equation*}
has a unique strong solution $ u\in L^{q+2}(\Omega\times[0,T]\times D)^6 $ with $ Mu\in L^{\frac{q+2}{q+1}}(\Omega\times[0,T]\times D)^6 $ if $ D\subset\R^3 $ is a bounded Lipschitz domain and $ \kappa:D\to\R^{6\times 6} $ is a uniformly bounded and uniformly elliptic matrix with measurable dependence in space. Their idea is to make a Galerkin approximation with respect to an arbitrary orthonormal basis $ (h_n)_n $ of $ H(\curl,0)(D)\times H(\curl)(D) $, that is also a basis of $ L^{2}(D)^6 $. However, besides many inaccuracies, they make two severe mistakes. 

Beginning from $ (11.12) $ on page $ 239 $, they derive 
\[ \int_0^{T}\langle (G\ast u_n)(s),u_n(s)\rangle_{L^{2}(D)^6}\operatorname{ds}\to \int_0^{T}\langle (G\ast u)(s),u(s)\rangle_{L^{2}(D)^6}\operatorname{ds} \]
for $ n\to\infty $ as a consequence of the weak convergences of $ G\ast u_n\to G\ast u $ and $ u_n\to u $ in $ L^{2}([0,T]\times D)^6 $ as $ n\to\infty. $ However, such an argument is not available in general. Moreover, they in their a priori estimate for the approximating problem, they implicitly use $$ \|\sum_{j=1}^{n}\langle u_0,h_j\rangle_{L^{2}(D)^6}h_j\|_{L^{2(q+1)}(D)^6}\leq C\|u_0\|_{L^{2(q+1)}(D)^6} $$ with a constant independent of $ n\in\N $. \ref{stochastic_maxwell_transformation_approx_equation_energy}. However, this is not true in general. As far as we know, such a result is only known on the torus, with the choice $ h_n(x)=e^{inx} $. This is one of the reasons, why we had to use the operators $ S_n $ that are also bounded on $ L^{p}(D)^6. $ \\ \\
Getting back to our result, we want to point out that the restriction to $ q\in (1,2] $ only comes from the H\"older estimate $$ \|F'(y_n)S_{n-1}\widetilde{b_j}\|_{L^{2}(\Omega\times[0,T]\times D)^6}\leq \|y_n\|^{2q}_{L^{q+2}([0,T]\times D)}\|S_{n-1}\widetilde{b_j}\|_{L^{\frac{2(q+2)}{2-q}}(\Omega\times[0,T]\times D)^6}^2$$
in the proof of Proposition \ref{stochastic_maxwell_transformation_approx_equation_energy}. Hence, if one assumes $ b_j\equiv 0$ one gets the same result as in Theorem \ref{stochastic_maxwell_main_result} for all $ q\in (1,\infty). $ In particular, this is true for the deterministic equation. Hence, we gave a correct proof for the theorem of Roach, Stratis and Yannacopoulus if $ \kappa\equiv I $ and $ D $ is a bounded $ C^1 $-domain or $ D=\R^3. $\\ \\
Next, we want to comment on the odd-looking condition $$ P_{n}\big(b_i(s)e^{-i\sum_{j=1}^{N}B_j\beta_j(s)}\big)=b_i(s)e^{-i\sum_{j=1}^{N}B_j\beta_j(s)} $$ from $ [\operatorname{M5}]$ for $ n\in\N $ large enough and for all $ s\in [0,T] $, $ i=1,\dots,N $ in case that $ q=2 $. We need it in the proof of Proposition \ref{stochastic_maxwell_transformation_approx_equation_energy} for the estimate $$ \|S_n\big(b_i(s)e^{-i\sum_{j=1}^{N}}B_j\beta_j(s)\big)\|_{L^{\infty}(D)^6}\leq C\|b_i(s)e^{-i\sum_{j=1}^{N}}B_j\beta_j(s)\|_{L^{\infty}(D)^6} $$ 
with a constant independent of $ n\in\N. $ It might be possible to get this inequality without our restrictive assumption in special cases. However, we want to point out, that even in the case $ D=\R^3 $ the boundedness of $ S_n $ on $ L^{\infty}(D)^6 $ is wrong, since it would imply the boundedness of the Hilbert transform $ L^{\infty}(D). $ If the $ B_j $ are constant, the assumption reduces to $ P_nb_i(s)=b_i(s). $ If $ D=\R^{3} $, this means that the Fourier transform $ \widehat{b_i}(s) $ is compactly supported in a timely independent set. In case that $ D $ is a bounded $ C^1 $-domain, this means, that $ b_i $ is of the form
\[ b_i(s)=\sum_{k=1}^{M} b_i^{(k)}(s)h_k \]
for some scalar valued $ b_i^{(k)}:\Omega\times [0,T]\to \C. $ Here, $ h_k=(h_{k,1},h_{k,2}) $ and $ h_{k,1} $ and $ h_{k,1} $ are eigenvectors of the operators $ A^{(1)} $ and $ A^{(2)} $, we introduced in Proposition \ref{stochastic_maxwell_hodge_laplacian_gaussian_bounds}.\\ \\
Last but not least, we want to discuss, why we did not treat coefficients in front of the Maxwell operator. Our approach is based on the interplay between $ -M^2 $, $ \Delta_H $ and the Helmholtz projection $ P_H $. In fact, we showed $ -M^2=\Delta_H $ on $ R(P_H) $ and $ -M^2=0 $ in $ N(P_H)=N(M). $ One might say, that we added a self-adjoint operator $ A=-\operatorname{grad}\Div $ with $ N(A)=R(P_H) $ to $ -M^2 $, such that the sum, namely $ -\Delta_H, $ has generalized Gaussian bounds. 
If we now replace $ M $ with 
\[ M_{\varepsilon,\mu}\binom{u_1}{u_2}= \binom{\varepsilon(x)^{-1}\curl u_2}{-\mu(x)^{-1}\curl u_1} \]
with the same perfect conductor boundary condition $ u_1\times\nu=0 $ on $ \partial D, $ we end up with
\[ -M^2\binom{u_1}{u_2}= \binom{\varepsilon(x)^{-1}\curl \mu(x)^{-1}\curl u_1}{\mu(x)^{-1}\curl \varepsilon(x)^{-1}\curl u_2} \]
with the boundary condition $ u_1\times\nu=0 $ and $ \big(\varepsilon^{-1}\curl u_2\big)\times\nu =0 $ on $ \partial D $ and uniformly bounded, positive definite  and hermitian $ \varepsilon,\mu:D\to\C^{3\times 3} $. The operator $ -M^2 $ is then positiv and self-adjoint with respect to a weighted scalar product on $ L^{2}(D)^6, $ namely
\[ \langle v,w\rangle_{\varepsilon,\mu}:=\int_{D}\varepsilon(x)v_1(x)\cdot w_1(x)\operatorname{dx}+\int_{D}\mu(x)v_2(x)\cdot w_2(x)\operatorname{dx}. \]
We need a weighted version of the Helmholtz projection $ P_{\varepsilon,\nu} $. We project orthogonally with respect to $ \langle\cdot,\cdot\rangle_{\varepsilon,\mu} $ onto \[\Big \{(u_1,u_2)\in L^{2}(D)^6: \Div(\varepsilon u_1)=0,\Div (\mu u_2)=0 \text{ and } (\mu u_2)\cdot\nu=0\text{ on }\partial D\Big \}. \] 
If we define  
$$ A_{\varepsilon,\mu}\binom{u_1}{u_2}=-\binom{\operatorname{grad}\Div(\varepsilon u_1) }{\operatorname{grad}\Div (\mu u_2)}, $$
one calculates that $ A_{\varepsilon,\mu} $ is symmetric with respect to $ \langle\cdot,\cdot\rangle_{\varepsilon,\mu}. $ Moreover, $ -M_{\varepsilon,\mu}^2 $, $ -M_{\varepsilon,\mu}^2+A_{\varepsilon,\mu} $ and $ P_{\varepsilon,\mu} $ have the same relationship as their counterparts with $ \varepsilon=\mu=I. $ 

Hence, one has to show that $ -M_{\varepsilon,\mu}^2+A_{\varepsilon,\mu} $ on the domain
\begin{align*}
\Big\{ &\curl u_1,\curl u_2, \curl\mu^{-1}\curl u_1,\curl\varepsilon^{-1}\curl u_2\in L^{p}(D)^3, \Div (\varepsilon u_1)\in W_0^{1,p}(D),\\
&\Div (\mu u_2)\in W^{1,p}(D),u_1\times\nu =0,(\mu u_2)\cdot\nu=0,(\varepsilon^{-1}\curl u_2)\times \nu=0\text{ on }\partial D \Big \}.
\end{align*}
has generalized Gaussian bounds, if one wants to generalize our result. However, even in case of smooth $ \varepsilon,\mu $ and $ \partial D $, such a result is unknown so far.

    \subsection{Acknowledgement}
I gratefully acknowledge financial support by the Deutsche Forschungs\-gemeinschaft (DFG) through CRC 1173. Moreover, I thank my advisor Lutz Weis and Roland Schnaubelt for many useful discussions and for pointing out references on the subject. I am also grateful, that Peer Kunstmann answered many questions about the Hodge Laplacian and about generalized Gaussian bounds. Last but not least, I want to mention the help of Fabian Hornung and Christine Grathwohl. They read the article carefully and gave many useful comments.\newpage
    \bibliographystyle{abbrv} 
    \bibliography{Literaturverzeichnis}

\end{document}